\documentclass[12pt,leqno,twoside]{amsart}
\usepackage[T1]{fontenc}
\usepackage[utf8]{inputenc}
\usepackage[babel=true]{csquotes}  
\usepackage{amstext,amsmath,amscd, bezier,indentfirst,amsthm,amsgen,enumerate}
\usepackage[left=2.5cm, right=2.5cm, top=2.5cm,bottom=2.5cm]{geometry} 
\usepackage{todonotes}
\usepackage{subcaption}
\usepackage[all,knot,arc,import,poly]{xy}
\usepackage{amsfonts,color}
\usepackage{amssymb}
\usepackage{latexsym}
\usepackage{epsfig}
\usepackage{graphicx}
\usepackage{srcltx}
\usepackage{enumitem}
\usepackage{accents}
\usepackage{imakeidx}
\usepackage{rotating}
\usepackage{verbatim}
\usepackage{fancyhdr}
\fancyheadoffset[RE,LO]{-0.01\textwidth}
\makeindex
\usepackage{tikz}
\usetikzlibrary{patterns}
\usetikzlibrary{matrix,arrows,decorations.pathmorphing}
\usepackage{tikz-cd}
\tikzset{commutative diagrams/.cd}
\usepackage{framed,lipsum}
\usepackage{float}
\usepackage{pgfplots}
\usetikzlibrary{patterns}
\usetikzlibrary{intersections}
\usetikzlibrary{positioning,shadows,arrows}
\tikzstyle{every node}=[anchor=west, minimum height=3em]
\usetikzlibrary{shapes,snakes}
\usepackage{wrapfig}
\definecolor{forestgreen}{rgb}{0.00, 0.39, 0.00} 
\definecolor{blueblue}{rgb}{0.40, 0.00, 1.00} 
\definecolor{sienna}{rgb}{0.33, 0.08, 0.11}

\newcommand{\defi}{\textbf}
\theoremstyle{plain}
\newtheorem{theorem}{Theorem}[section] 

\newtheorem*{theorem*}{Theorem}
\newtheorem*{theoremfr*}{Théorème}
\newtheorem*{hypothesisfr*}{Hypothèse}
\newtheorem*{hypothesis*}{Hypothesis}
\newtheorem{lemma}[theorem]{Lemma}
\newtheorem{proposition}[theorem]{Proposition}
\newtheorem{corollary}[theorem]{Corollary}

\theoremstyle{definition}
\newtheorem{definition}[theorem]{Definition}
\newtheorem*{definition*}{Definition}
\theoremstyle{remark}
\newtheorem{remark}[theorem]{\sc Remark}
\newtheorem*{question*}{\sc Question}
\newtheorem*{remark*}{\sc Remark}
\newtheorem*{remarkfr*}{\sc Remarque}
\newtheorem*{examplefr*}{\sc Exemple}
\newtheorem*{example*}{\sc Example}

\newtheorem{example}[theorem]{\sc Example}

\def\bR{{\mathbb R}}

\def\Sing{{\rm Sing}}

\def\const.{{\rm const.}}

\def\Int{{\rm Int\ }}

\usepackage[style=alphabetic,natbib=true,backref=true,backend=bibtex]{biblatex}
\usepackage{mathscinet}
\usepackage[pdftex, pdfusetitle, plainpages=false,  bookmarks, bookmarksnumbered,colorlinks, linkcolor=blue, citecolor=red,filecolor=black, urlcolor=black]{hyperref}
\usepackage{pgfplots}
\pgfplotsset{width=7cm,compat=1.8}
\usepackage[tight]{minitoc} 
\usepackage{caption}
\makeatletter
\renewcommand*{\numberline}[1]{\hb@xt@1em{#1\hfil}} 
\makeatother

\usepackage[style=alphabetic,natbib=true,backref=true,backend=bibtex]{biblatex}
\usepackage{mathscinet}

\keywords{strict local minimum, Poincaré-Reeb tree, non-convexity, level curve, stabilisation, real algebraic curve, polar curve, star domain, smooth Jordan curve}
\usepackage{pgfplots}

\pgfplotsset{width=7cm,compat=1.8}

\usepackage{caption}

\usepackage[foot]{amsaddr}

\makeatletter
\renewcommand{\email}[2][]{%
  \ifx\emails\@empty\relax\else{\g@addto@macro\emails{,\space}}\fi%
  \@ifnotempty{#1}{\g@addto@macro\emails{\textrm{(#1)}\space}}%
  \g@addto@macro\emails{#2}%
}
\makeatother

\begin{document}

\title{\textsc{Measuring the local non-convexity of real algebraic curves}}
\author{Miruna-\c Stefana Sorea}
\address{\href{https://www.mis.mpg.de/de.html}{Max-Planck-Institut für Mathematik in den Naturwissenschaften, Leipzig, Germany}}
\email{miruna.sorea@mis.mpg.de}

\maketitle

\section*{Abstract}\label{sect:CrestsValleys}

The goal of this paper is to measure the non-convexity of compact and smooth connected components of real 
       algebraic plane curves. We study these curves first in a general setting and then in an asymptotic one. In particular, we consider sufficiently small levels of a real bivariate polynomial in a small enough neighbourhood of a strict local minimum at the origin of the real affine plane. We introduce and describe a new combinatorial object, called the Poincaré-Reeb graph, whose role is to encode the shape of such curves and allow us to quantify their non-convexity. Moreover, we prove that in this setting the Poincaré-Reeb graph is a plane tree and can be used as a tool to study the asymptotic behaviour of level curves near a strict local minimum. Finally, using the real polar curve, we show that locally the shape of the levels stabilises and that no spiralling phenomena occur near the origin.

\section*{Introduction}
This work is situated at the crossroads between the geometry and topology of singularities and the study of real algebraic plane curves. The origins of this subject date back to the works of Harnack, Klein, and Hilbert (\cite{Har1}, \cite{Kl1}, \cite{Hil1}). Continuous progress and recent interest on these subjects is shown for instance by Coste, de la Puente, Dutertre, Ghys, Itenberg, Sturmfels, and Viro (\cite{CP}, \cite{ND}, \cite{Gh1}, \cite{II1},\cite{Stu1}, \cite{Vi1}). 

Our goal is to measure how far from being convex a given algebraic curve is and what is its behaviour near a strict local minimum. Non-convexity plays an important role in many applications, such as optimisation theory or computational geometry (see for instance \cite{JK17} or \cite{MR3937812}). To this end, we construct and study a new combinatorial object called the Poincaré-Reeb tree, whose role is to encode the shape of a smooth compact connected component of a real algebraic plane curve. In particular, we will use this tool to quantify the non-convexity of the curve. The Poincaré-Reeb tree is adapted from the classical construction introduced by H. Poincaré (see \cite[1904, Fifth supplement, page 221]{Po}), which was rediscovered by G. Reeb in 1946 (in \cite{Re}). Both used it as a tool in Morse theory. Namely, given a Morse function on a closed manifold, they associated it a graph as a quotient of the manifold by the equivalence relation whose classes are the connected components of the levels of the function. We will perform an analogous 
      construction for a special type of manifold with boundary. We prove that in our setting the Poincaré-Reeb graph is
a plane tree, and show several other properties of this tree.

Note that there exist results on graphs encoding the topology of real algebraic plane curves (see, for example, \cite{CLPPRT} and the references therein). Those graphs are different from our novel Poincaré-Reeb construction.

We will now list the main contributions of this paper.

\begin{figure}[H]
\centering
\includegraphics[scale=0.19]{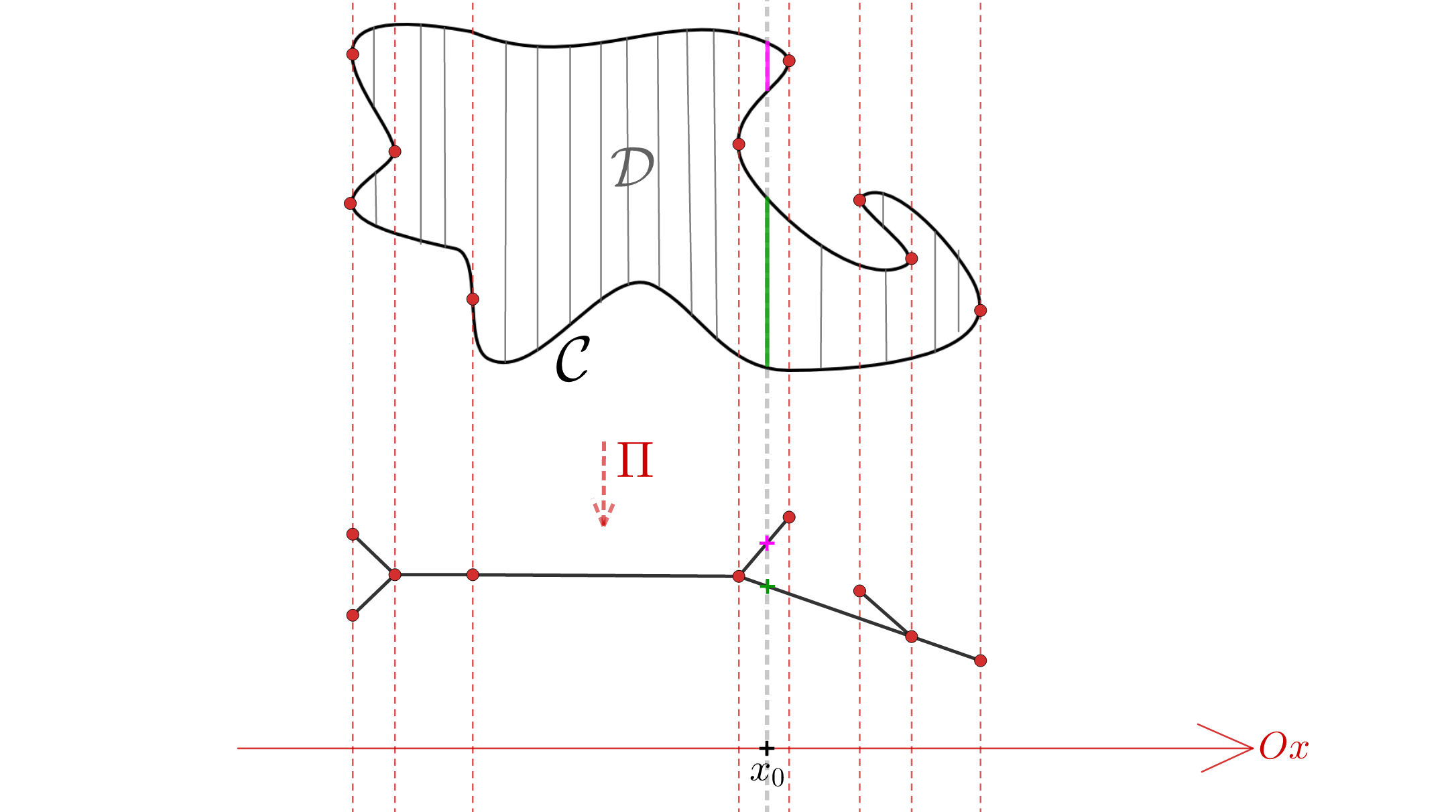} 
\caption{The Poincaré-Reeb graph of a smooth and compact connected component of a real algebraic curve in $\mathbb{R}^2$.\label{fig:jordanReeb}}
\end{figure}

\textbf{Measuring the non-convexity via the The Poincaré-Reeb tree.} Consider a compact and smooth connected component $\mathcal{C}$ of a real algebraic curve contained in the plane $\mathbb{R}^2$ (see Figure \ref{fig:jordanReeb}). Denote by $\mathcal{D}$ the disk bounded by $\mathcal{C}$. Endow the plane $\mathbb{R}^2$ with its canonical orientation and with its foliation by vertical lines. The curve $\mathcal{C}$, being a connected component of an  algebraic curve, has only finitely many points of vertical tangency. The vertices of the Poincaré-Reeb graph are images of the points of the curve having vertical tangent by the projection 
       $\Pi : \mathcal{D} \to \mathbb{R}$, $\Pi(x,y):= x$. For $x_0\in\mathbb{R},$ if the fibre $\Pi^{-1}(x_0)$ is not the  empty set, then it is a finite union of vertical segments. We define an equivalence relation that lets us contract each of these segments into a point. By the quotient map, the image of the initial disk becomes a one-dimensional connected topological subspace embedded in the real plane. More precisely, it is a connected plane graph. In addition to this construction, we prove the following result:

\begin{theorem}[see Corollary \ref{cor:PReebInGeneral}]\label{th:transverseTree}
The Poincaré-Reeb graph associated to 
a compact and smooth connected component of a real 
       algebraic plane curve and to
a direction of projection $x$ is a plane tree whose open edges are transverse to the foliation induced by the function $x$. Its vertices are endowed with a total preorder relation induced by the function $x$.
\end{theorem}

Note that the main tool used in the proof of Theorem \ref{th:transverseTree} is the integration of constructible functions with respect to the Euler characteristic (see \cite[Chapter 3, page 22]{Co}).

$ \\ $

\textbf{Nested level curves of a real bivariate polynomial near a strict local minimum.} Let us consider a polynomial function $f:\mathbb{R}^2\rightarrow \mathbb{R}$, such that $f$ has a strict local minimum at the origin and $f(0)=0$. We focus on the so-called real Milnor fibres at the origin of $f$, namely the small enough level curves $(f(x,y)=\varepsilon),$ for sufficiently small $0<\varepsilon\ll 1.$ We prove that they are smooth Jordan curves (see Lemma \ref{lemma:origInInterior}). The simplest example is the circle. Whenever the origin is a Morse strict local minimum, we show that the small enough levels become boundaries of convex topological disks. In the non-Morse case, these curves may fail to be convex, as was shown by a counterexample found by Michel Coste.  Let us restrict the study of the Poincaré-Reeb trees to the asymptotic setting, namely let us study the trees associated to small enough level curves of a real bivariate polynomial function, near a strict local minimum at the origin of the plane. We show that the tree stabilises for sufficiently small level curves. We call it the asymptotic Poincaré-Reeb tree. Therefore the shape of the levels stabilises near the strict local minimum (see Figure \ref{fig:stabShapes}).

\begin{figure}[H]
\centering
\includegraphics[scale=0.3]{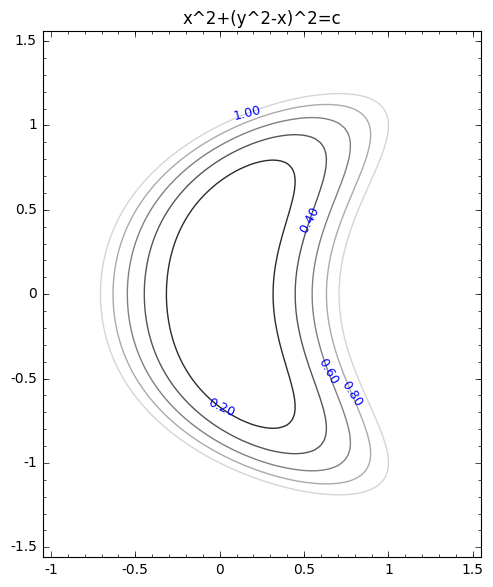} 
\caption{Stabilisation of the shape of the levels.\label{fig:stabShapes}}
\end{figure}

We also prove the following result:

\begin{theorem}[see Theorem \ref{th:asymptoticPReeb}]
The asymptotic Poincaré-Reeb tree is a rooted tree and the total preorder relation on its vertices, induced by the function $x$, is strictly monotone on each geodesic starting from the root. 
\end{theorem}

The strict monotonicity on the geodesics starting from the root allows us to prove that the small enough level curves $\mathcal{C}_\varepsilon$ have no turning back or spiralling phenomena. 

In other words, near a strict local minimum, for $0<\varepsilon\ll 1$ sufficiently small there are no level curves $(f=\varepsilon)$ like the one in Figure \ref{fig:spiralNo}.

\begin{figure}[H]
\centering
\includegraphics[scale=0.099]{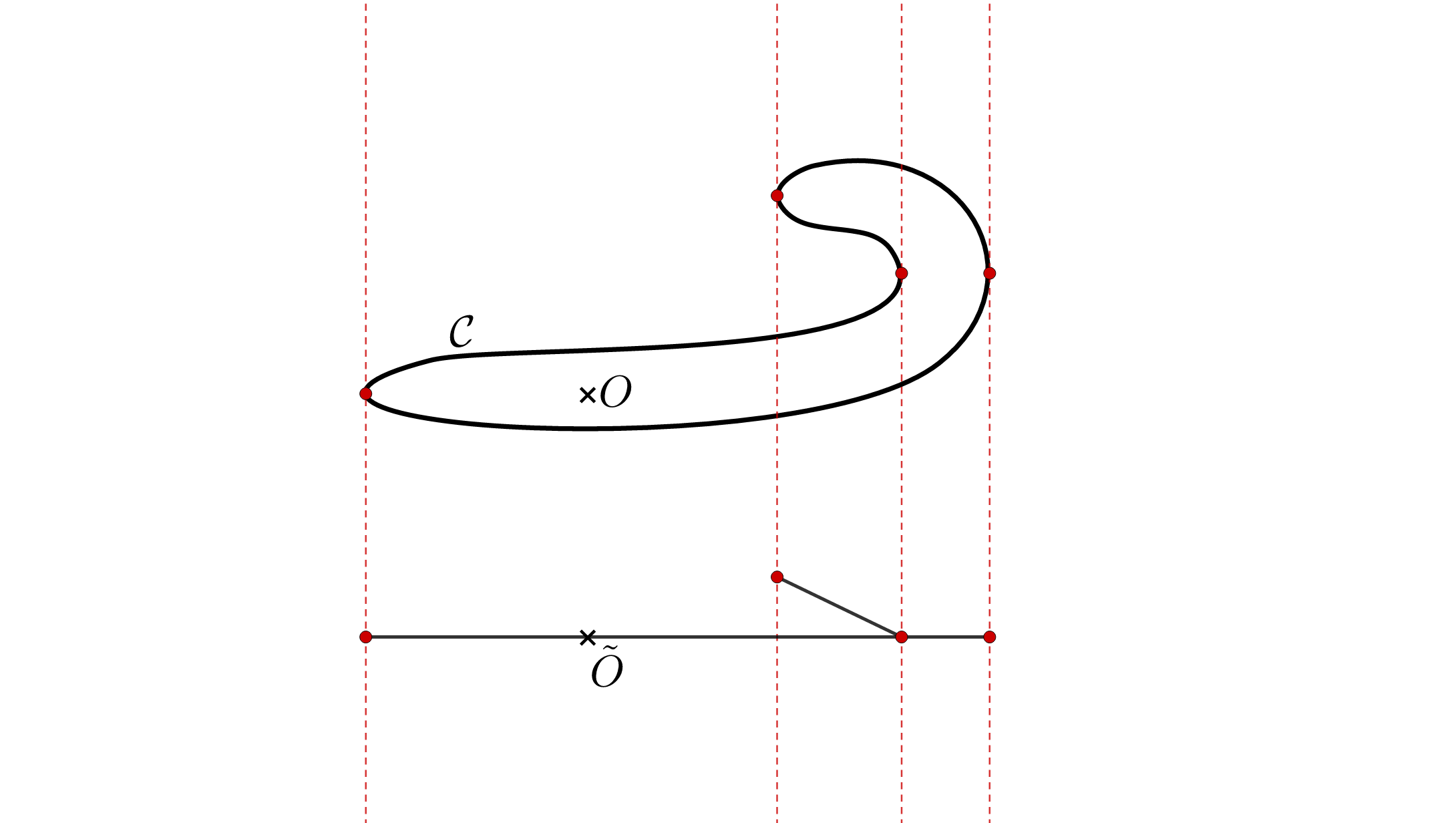} 
\caption{Forbidden shape for small enough $0<\varepsilon\ll 1$: there are no going backwards like this.\label{fig:spiralNo}}
\end{figure}

The present paper is organised as follows. After introducing the necessary notations and hypotheses in Section \ref{SectHypothesis}, we show in Section \ref{sect:convex} that topological disks bounded by the level curves of $f$ in a small enough neighbourhood of a \emph{Morse} strict local minimum are convex sets (see Theorem \ref{th:convexNondegenerated}). Section \ref{sect:costeGener} is dedicated to Coste's counterexample (see Example \ref{ex:Coste}) in the \emph{non-Morse} case and our generalisations. In Section \ref{sect:ConstructionPR} we give the construction of a new combinatorial object, called the Poincaré-Reeb graph, whose role is to measure the non-convexity of a compact smooth connected component of a plane algebraic curve in $\mathbb{R}^2$. Using a Fubini-type theorem and integration with respect to the Euler characteristic, we show that the Poincaré-Reeb graph is in fact a plane tree (see Corollary \ref{cor:PReebInGeneral}). Section \ref{sect:asymp} focuses on the asymptotic case, where the real polar curve plays a key role. We prove that in a small enough neighbourhood of the origin the Poincaré-Reeb trees become equivalent, starting with a small enough level. In other words, the shape of the levels stabilises near the origin. In addition, the Poincaré-Reeb tree allows us to prove Theorem \ref{th:asymptoticPReeb}. Namely we show that asymptotically there is no spiralling phenomenon near the strict local minimum. This is due to the fact that the induced preorder endowing the set of vertices is strictly monotonous on each geodesic starting from the root of the tree.

\begin{figure}[H]
\centering
\includegraphics[scale=0.13]{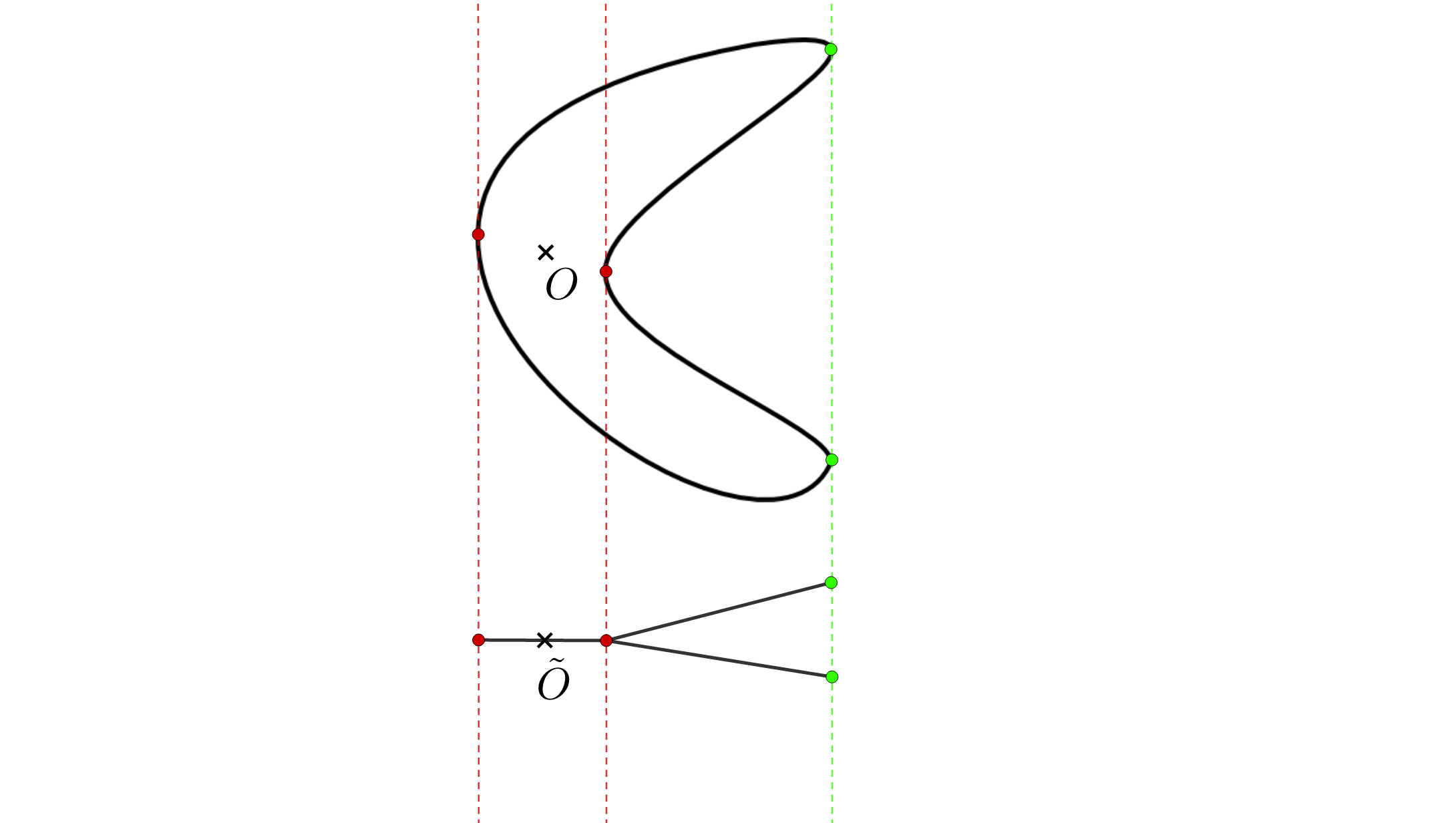} 
\caption{An asymptotic Poincaré-Reeb tree.\label{fig:asymptotic}}
\end{figure}

\section*{Acknowledgements}
This paper represents a part of my PhD thesis (see \cite{So1}), defended at \href{http://math.univ-lille1.fr/}{\textit{Paul Painlevé Laboratory}}, \href{https://www.univ-lille.fr/}{\textit{Lille University}} and financially supported by \href{http://math.univ-lille1.fr/~cempi/}{\textit{Labex CEMPI}} (ANR-11-LABX-0007-01) and by \href{http://www.hautsdefrance.fr/}{\textit{Région Hauts-de-France}}. I would like to express my gratitude towards my PhD advisors, \href{http://math.univ-lille1.fr/~bodin/index.html}{Arnaud Bodin} and \href{http://math.univ-lille1.fr/~popescu/}{Patrick Popescu-Pampu}, for their guidance throughout this work. I would like to thank \href{http://ergarcia.webs.ull.es/}{Evelia R. Garc\' ia Barroso} and \href{https://webusers.imj-prg.fr/~ilia.itenberg/}{Ilia Itenberg} for reviewing my PhD thesis and \href{http://perso.ens-lyon.fr/ghys/accueil/}{\' Etienne Ghys} for being the president of the committee. 

\section{Notations and hypotheses}\label{SectHypothesis}

In this chapter, the expressions \enquote{sufficiently small $\varepsilon$} or \enquote{small enough $\varepsilon$} mean: \enquote{there exists $\varepsilon_{0}>0$ such that, for any $0<\varepsilon<\varepsilon_{0}$, one has \ldots}. We will often denote this by $0<\varepsilon\ll 1.$

We will consider a polynomial function $f:\bR^2\rightarrow \bR$ in two variables with $f(0,0)=0$. Let us \textbf{fix a neighbourhood} $V$ of $(0,0)$ such that:

- the point $(0,0)$ is \textbf{the only strict local minimum} of $f$ in $V$;

- we have $V\cap (f=0)=\{(0,0)\}.$

\begin{definition}\label{def:levelCEpsilon}
Let us consider a polynomial function $f:\mathbb{R}^2\rightarrow\mathbb{R}$ that vanishes at the origin $O$, exhibiting a strict local minimum at this point. For $\varepsilon>0$, consider the set $f^{-1}([0,\varepsilon])$ and denote \defi{its connected component that contains the origin} by $\mathcal{D}_\varepsilon$. Denote by $\mathcal{C}_\varepsilon:=\mathcal{D}_\varepsilon\setminus \Int \mathcal{D}_\varepsilon.$ 

\end{definition}

\begin{figure}[H]
\centering

\includegraphics[scale=0.19]{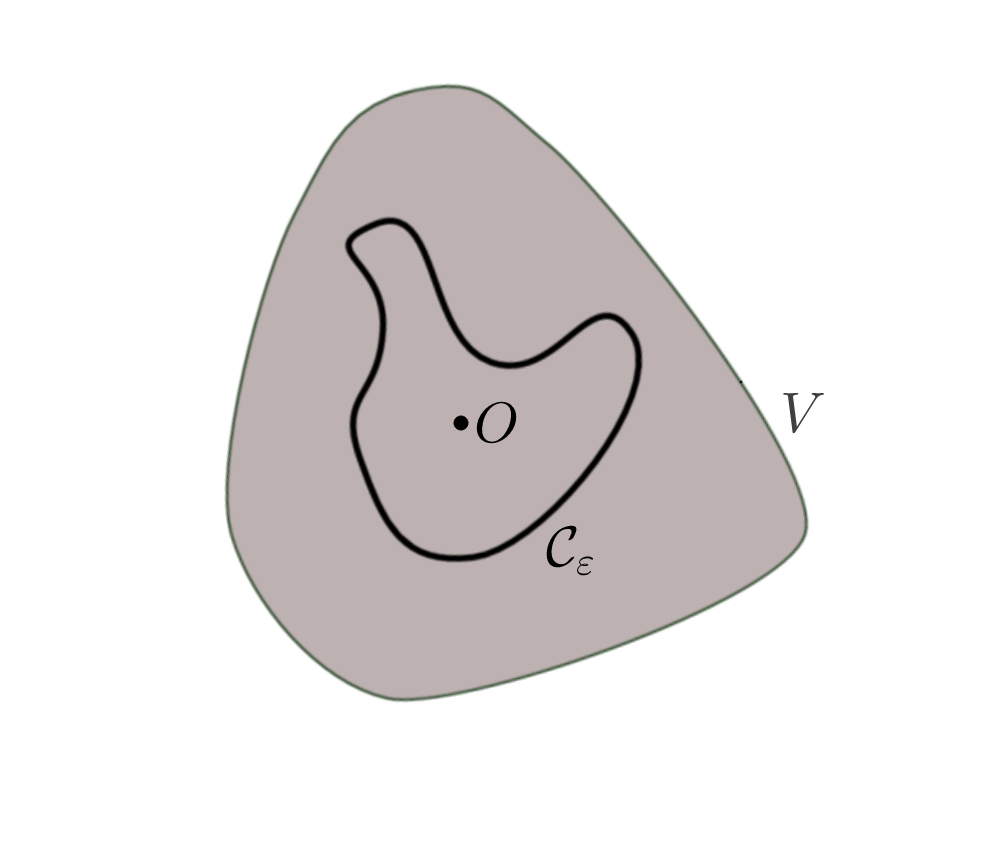} 
\caption{The curve $\mathcal{C}_\varepsilon$.\label{fig:izolat}}
\end{figure}

Later on (see Lemma \ref{lemma:origInInterior}), we shall prove that for $0<\varepsilon\ll 1$ sufficiently small, $\mathcal{C}_\varepsilon$ is diffeomorphic to a circle and that $\mathcal{D}_\varepsilon$ is diffeomorphic to a disk, completely included in $V$.

\begin{example}\label{ex:acnode}
The surface $\mathrm{graph}(f):=\{(x,y,z)\in\mathbb{R}^3 \mid z=f(x,y)\}$ of the cubic $f(x,y):=y^2-x^3+x^2$ is shown in Figure \ref{fig:acnode3D} (see for instance \cite[page 57, Figure 2.2]{Wk}).

\begin{figure}[H]
\centering
\includegraphics[scale=0.19]{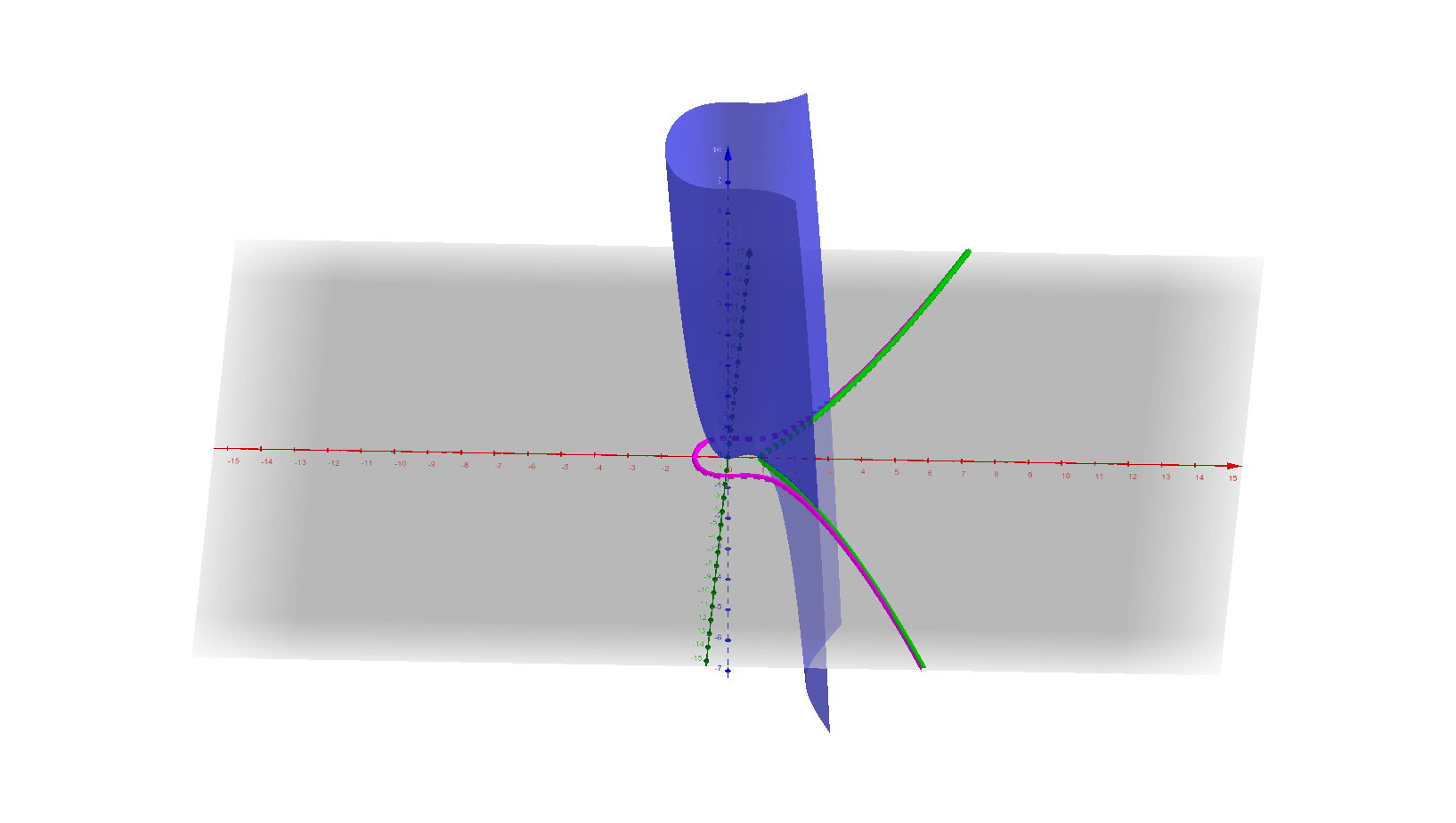} 
\caption{The graph of the cubic $f(x,y):=y^2-x^3+x^2$.\label{fig:acnode3D}}
\end{figure}

We did not choose yet a sufficiently small neighbourhood $V$ of the origin. An illustration of such a neighbourhood  $V$ that satisfies the hypotheses from Section \ref{SectHypothesis} is shown in Figure \ref{fig:cubic} below. The zero locus of $f$ has two connected components: a point and a curve. We take $V$ small enough such that it does not intersect the curve, that is the connected component far from the origin.

\begin{figure}[H]
\centering
\includegraphics[scale=0.15]{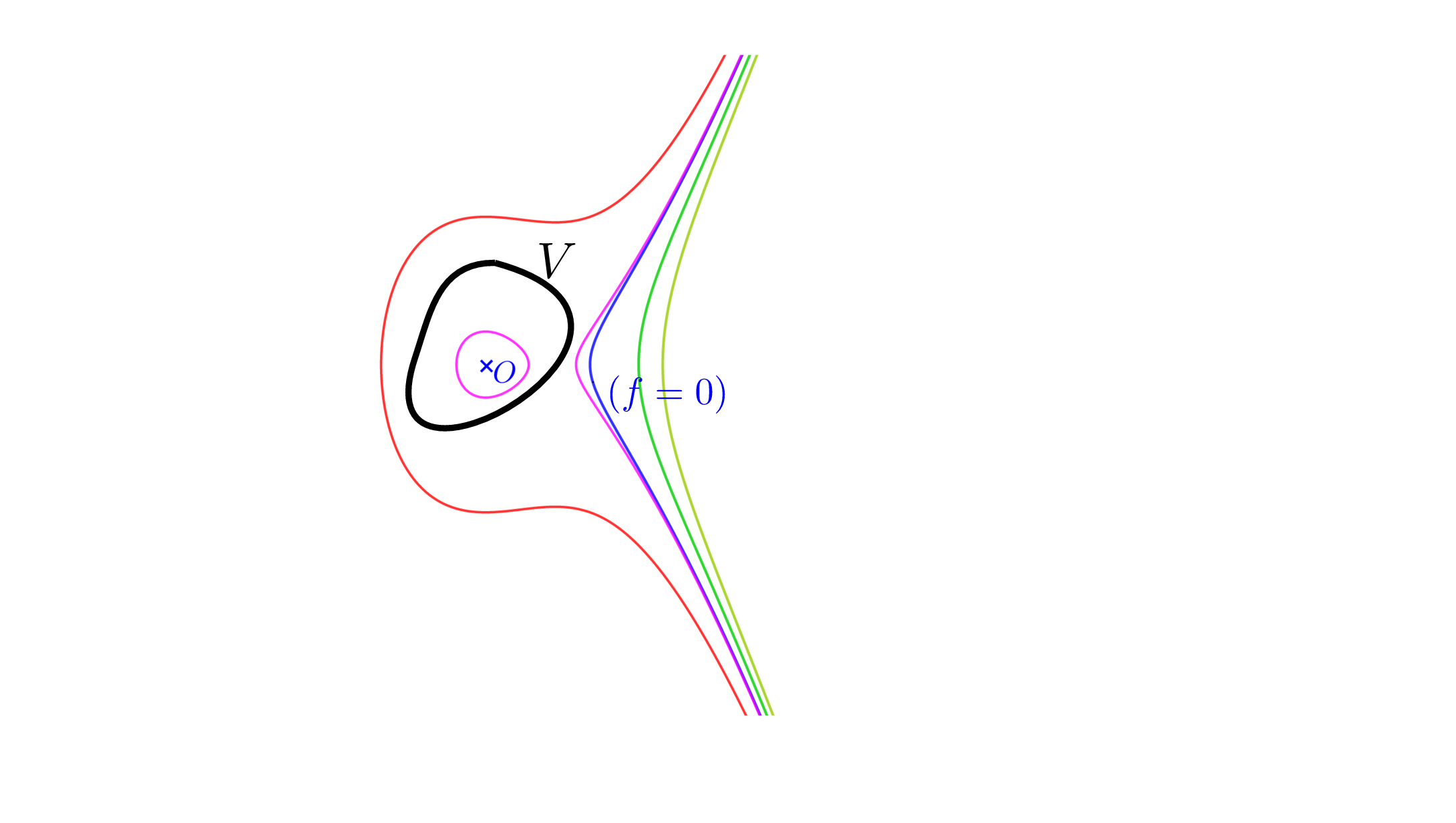} 
\caption{Choose a small enough neighbourhood $V$ of the origin such that the point $(0,0)$ is the only strict local minimum of $f$ in $V$ and $V\cap (f=0)=\{(0,0)\}.$\label{fig:cubic}}
\end{figure}

\end{example}

Examples of functions $f$ as in Section \ref{SectHypothesis} are shown in Figure \ref{fig:convexStrictMin}, Figure \ref{fig:banana3D}, Figure \ref{fig:boneLevelsSets}.

For the following classical definition see for instance \cite[page 10]{Mi2}:
\begin{definition}\label{def:SingularPoints}
Let $\varphi:\mathbb{R}^n\rightarrow \mathbb{R}^p,$ with $n,p\in\mathbb{N}^*$, $n\geq p,$ be a polynomial map.

 Let $\rho:=\mathrm{max}\{\mathrm{rank}\ \mathrm{Jac}(\varphi)(x)\mid x\in \mathbb{R}^n\}$ be the maximal rank of $\varphi$, where $\mathrm{Jac}(\varphi)(x)$ is the Jacobian matrix of $\varphi,$ evaluated at $x.$ 
 
Then the set of \defi{singular points}\index{singular point} (called also \defi{critical points}\index{critical point}) of $\varphi$ is by definition $$\mathrm{Sing}(\varphi):=\{x\in \mathbb{R}^n\mid \mathrm{rank}\ \mathrm{Jac}(\varphi)(x)<\rho\}.$$ The image of a singular point under $\varphi$ is called a \defi{singular value} of $\varphi$\index{singular value}.

 A point $x\in\mathbb{R}^n$ is called \defi{regular}\index{regular point} of $\varphi$ if it is not singular. 

\end{definition}

\begin{example}
Let $f:\bR^2\rightarrow\bR$ be a polynomial function. Then by Definition \ref{def:SingularPoints}, the set of singular points of $f$ is 
$\Sing f=\left \{(x,y)\in\bR^2\;\middle|\; \frac{\partial f}{\partial x}(x,y)=\frac{\partial f}{\partial y}(x,y)=0 \right \}.$

For more examples, see for instance \cite{PPParxiv} or \cite[page 57]{Wk}.
\end{example}

\section{Morse extrema and convex level sets}\label{sect:convex}

The classical well-known statement that follows is meant to present equivalent characterisations of convex functions (see, for instance, \cite[pages 640-650]{RRR}).

\begin{theorem}\label{th:convexProp}
Let us consider a convex set $C\subset\mathbb{R}^2$ and a $\mathcal{C}^2$ function $f:C\rightarrow\mathbb{R}$. The following properties are equivalent:

(1) for any two points $x$, $y\in C$, for any $t\in [0,1],$ we have $$f\left (tx+(1-t)y\right)\leq tf(x)+(1-t)f(y);$$

(2) the set $\mathrm{epi}(f):=\{(x,y,z)\in\mathbb{R}^3\mid z\geq f(x,y)\}$ is convex;

(3) the Hessian matrix of $f$, denoted by $\mathrm{Hess}_f(x,y)$, is everywhere  positive definite or positive semidefinite.
\end{theorem}

\begin{definition}
A $\mathcal{C}^2$ function $f:\mathbb{R}^2\rightarrow\mathbb{R}$ is \defi{convex}\index{convex function} if $f$ verifies one of the properties from Theorem \ref{th:convexProp}.
\end{definition}

\begin{theorem}\label{th:convexNondegenerated}
Let $f:\mathbb{R}^2\rightarrow \mathbb{R}$ be a polynomial function of the form $f(x,y)=ax^2+by^2+h(x,y),$ where $a,b\in\mathbb{R},$ $a>0,$ $b>0,$ $f(0,0)=0$ and $h:\mathbb{R}^2\rightarrow\mathbb{R}$ is a polynomial function of degree of each monomial at least $3$. Then for a sufficiently small $0<\varepsilon\ll 1$, $\mathcal{D}_\varepsilon$ is a convex set in a small enough neighbourhood of the origin.
\end{theorem}

\begin{proof}
Note that we are situated in a small enough neighbourhood of the origin. We have $a>0$, $b>0$, $\lim_{(x,y)\rightarrow (0,0)}\mathrm{det}H_f(x,y)=ab>0,$ for any $(x,y)$ in a small enough neighbourhood of $(0,0)$. Thus the Hessian matrix is positive definite and we conclude that $f$ is a convex function, since we are in dimension $2$ and by Sylvester criterion if a symmetric matrix has all diagonal elements and all the leading principal determinants positive, then the matrix is a positive definite matrix (see \cite[page 45]{Gi}, or \cite{RRR}). Recall the fact that positive definite forms form an open set in the space of all forms. Since the function $f$ is $\mathcal{C}^2$, its Hessian form varies continuously. 

Therefore, the epigraph of $f$ restricted to a convex domain  is a convex set. Let us consider $0<\varepsilon\ll 1,$ sufficiently small. Since $\tilde{\mathcal{D}}_\varepsilon:=\mathrm{epi}(f)\cap (z=\varepsilon),$ we obtain that $\tilde{\mathcal{D}}_\varepsilon$ is a convex set, since it is the intersection of two convex sets. Thus, the projection of $\tilde{\mathcal{D}}_\varepsilon$ on the plane $xOy$ is also convex, in fact it is equal to $\mathcal{D}_\varepsilon$. 

Let us sketch a second proof. Since $\lim_{x,y\rightarrow (0,0)}\mathrm{det}H_f(x,y)=ab\neq 0,$ for any $(x,y)$ in a small enough neighbourhood of $(0,0)$, the Hessian curve of $f$ given by $\{(x,y)\in\mathbb{R}^2\mid \mathrm{det}H_f(x,y)=0\}$, is empty. Thus (see \cite[page 71]{Wa}) for a sufficiently small $0<\varepsilon\ll 1,$ the level curve $\mathcal{C}_\varepsilon$ has no inflexion points, hence it is convex.

\end{proof}

\begin{example}
The circular paraboloid $z=x^2+y^2$ (see in Figure \ref{fig:convexStrictMin} below): $\mathcal{C}_\varepsilon$ are circles.

\begin{figure}[H]
\begin{center}
\includegraphics[scale=0.22]{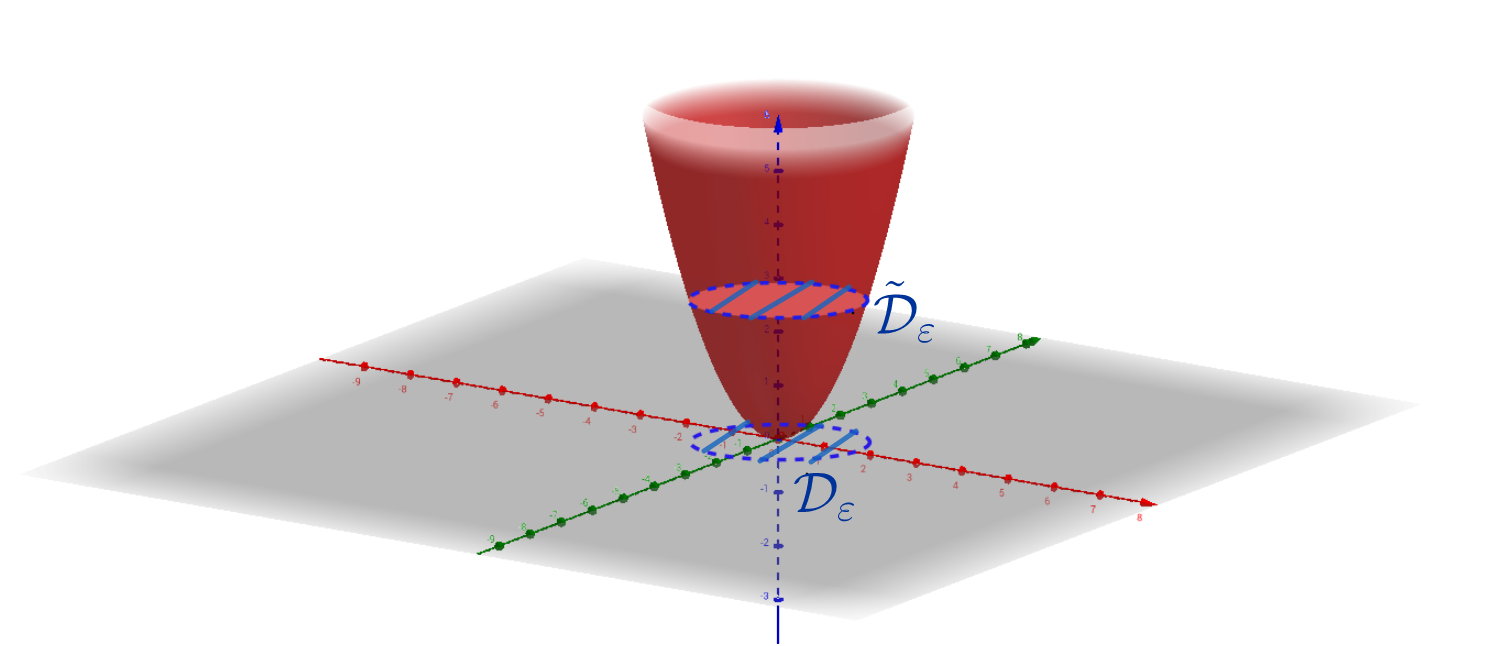} 
\end{center}
\caption{A non-degenerated (namely, Morse) critical point: we obtain convex disks.\label{fig:convexStrictMin}}

\end{figure}
\end{example}

\section{Non-convex level sets}\label{sect:costeGener}

The starting point of this research is the following question: 

\begin{question*}[E. Giroux to P. Popescu-Pampu, 2004]
Are the small enough level curves of $f$ near strict local minima always boundaries of convex disks (even for non-Morse singularities)? 
\end{question*}

\noindent \textsc{Answer}: The answer to this question is negative, as the following counterexample by M. Coste shows.

\begin{example}\label{ex:Coste}
Let us study the counterexample found by M. Coste: $$f(x,y):=x^2+(y^2-x)^2.$$
See Figure \ref{fig:bananaFamily} for the family of its level sets. Figure \ref{fig:banana3D} illustrates the graph of $f$ and Figure \ref{fig:bananaOneLevel} a single level $(f(x,y)=0.1).$

The idea used by M. Coste was that instead of taking the sum of squares of functions defining transverse curves like in $x^2+y^2$ (i.e. the lines defined by $x=0$ and $y=0$ are transverse), to take the sum of squares of $x$ and $y^2 - x$, which define two curves that are tangent at the origin.

Note that the function $f$ was also mentioned in \cite{JK}, \cite{Bo}.

\end{example}

\begin{figure}[H]
\centering
\begin{subfigure}[b]{0.45\textwidth}
\includegraphics[scale=0.3]{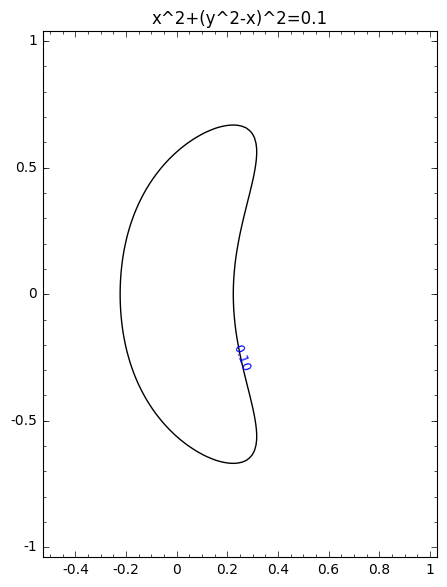} 
\caption{The level curve $(f(x,y)=0.1)$ in $\mathbb{R}^2$, where the function $f(x,y):=x^2+(y^2-x)^2$.}
\label{fig:bananaOneLevel}
\end{subfigure}
\begin{subfigure}[b]{0.45\textwidth}
\includegraphics[scale=0.3]{tmp_Bplcwq.png} 
\caption{A family of level curves in $\mathbb{R}^2$ of $f(x,y):=x^2+(y^2-x)^2$.}
\label{fig:bananaFamily}
\end{subfigure}
\caption{The counterexample provided by Coste.\label{fig:counterexC}}
\end{figure}

\begin{figure}[H]
\centering

\includegraphics[scale=0.185]{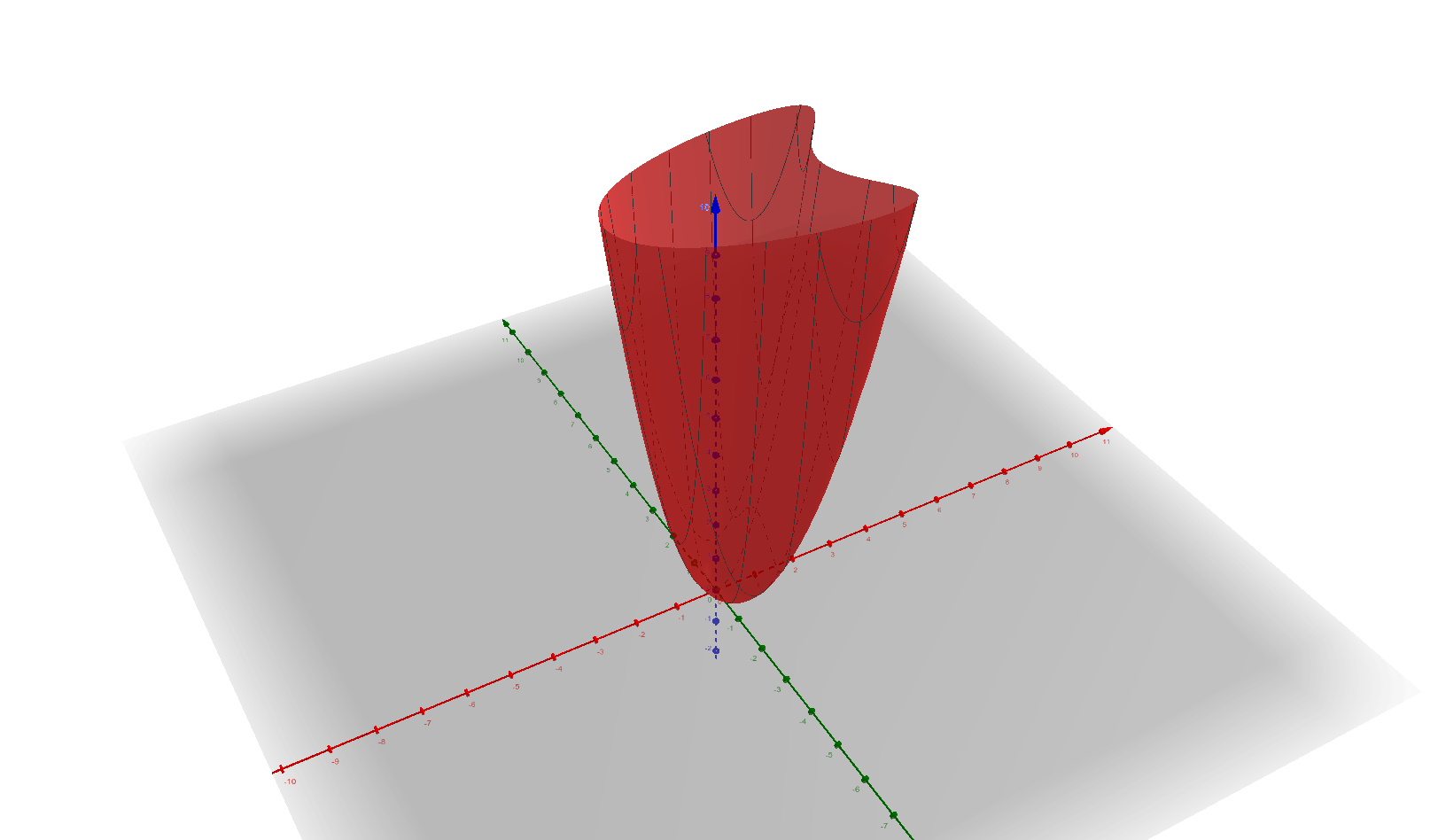} 
\caption{The graph of Coste's example in $\mathbb{R}^3$. \label{fig:banana3D}}
\end{figure}

\begin{proposition}\label{prop:NonConvex}

The topological disks $\mathcal{D}_{\varepsilon}$ from Example \ref{ex:Coste} are never convex for $\varepsilon>0$.
\end{proposition}

\begin{proof}
Indeed, if we denote by $M_\varepsilon
:=\mathcal{C}_\varepsilon\cap \{x>0\mid y=0\}$, by $P_\varepsilon:=
\mathcal{C}_\varepsilon \cap \{y>0\mid y^2-x=0\}$ and by $N_\varepsilon:=
\mathcal{C}_\varepsilon \cap \{y<0\mid y^2-x=0\}$, then we get $x_{M_\varepsilon}=
\sqrt{\frac{\varepsilon}{2}}$,
$y_{M_\varepsilon}=0$, 
$x_{P_\varepsilon}=
\sqrt{\varepsilon}$, 
$y_{P_\varepsilon}=\sqrt[4]{\varepsilon},$ and $x_{N_\varepsilon}=
\sqrt{\varepsilon}$, 
$y_{N_\varepsilon}=-\sqrt[4]{\varepsilon}$ (see Figure \ref{fig:nonconvex}).

\begin{figure}[H]
\centering
\includegraphics[scale=0.15]{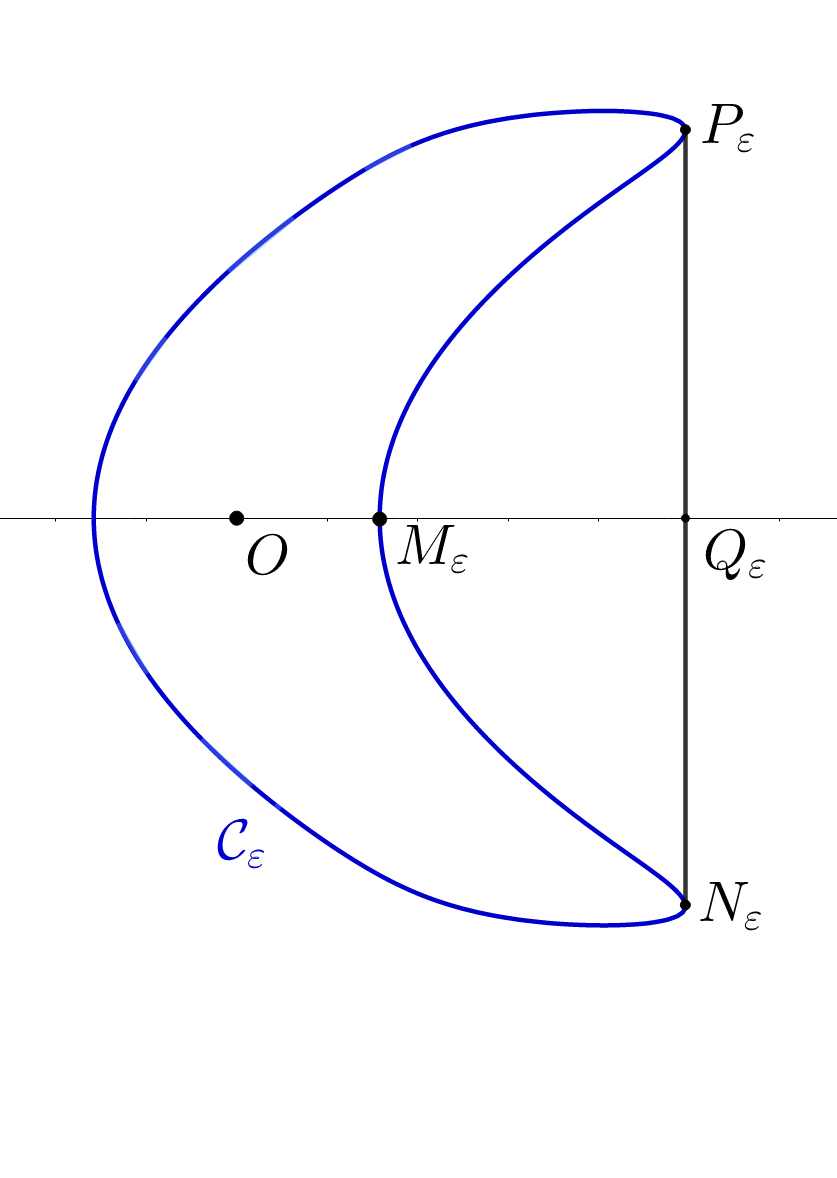} 
\caption{Coste's example is non-convex.\label{fig:nonconvex}}
\end{figure}

Denote by $Q_\varepsilon$ the midpoint of $[N_\varepsilon P_\varepsilon]$. Then $x_{Q_\varepsilon}=\sqrt{\varepsilon}>x_{M_\varepsilon}$. Since $y_{M_\varepsilon}=y_{Q_\varepsilon}=0$ and there is no other point to the right of $M_\varepsilon$, we conclude that $Q_\varepsilon\not \in \mathcal{D}_\varepsilon.$ Hence there exist the points $N_\varepsilon\in \mathcal{D}_\varepsilon$ and $P_\varepsilon\in \mathcal{D}_\varepsilon$ such that the segment $[N_\varepsilon P_\varepsilon]$ is not included in $\mathcal{D}_\varepsilon.$ In conclusion, we proved that there exists at least a point outside the disk $\mathcal{D}_\varepsilon$ on the segment $[N_\varepsilon P_\varepsilon]$.
\end{proof}

\subsection{Generalisations}
In the sequel we give some new examples of functions $f:\bR^2\rightarrow \bR$ with a strict local minimum at the origin $(0,0)$, whose level sets $(f(x,y)=\varepsilon)$ are all boundaries of non-convex disks for a sufficiently small $0<\varepsilon\ll 1$.

The shape of these level curves $\mathcal{C}_\varepsilon$ stabilises for sufficiently small $0<\varepsilon\ll 1$, as we shall prove in this paper.

\begin{example}\label{ex:unu}

Let us take the following polynomial:
$$f(x,y):=x^{16}+(y^2+x)^2 (y^2-x)^2.$$
For a family of level curves of $f$, see Figure \ref{fig:boneLevelsSets}.

\begin{figure}[H]
\centering
\includegraphics[scale=0.4]{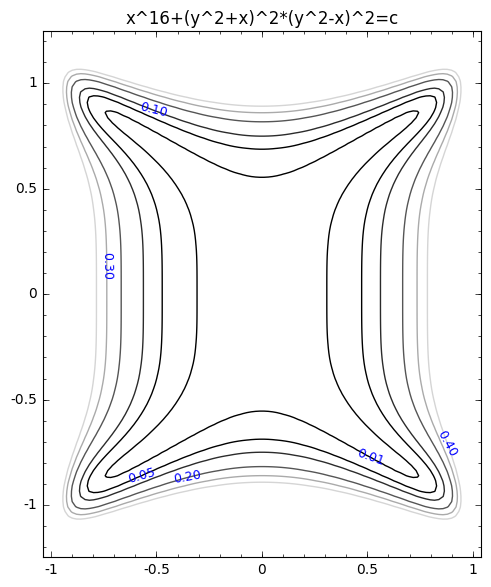} 
\caption{A family of level curves in $\mathbb{R}^2$ of $f(x,y):=x^{16}+(y^2+x)^2 (y^2-x)^2$.\label{fig:boneLevelsSets}}
\end{figure}

\end{example} 
  
\begin{example}\label{ex:doi}
Let us consider the following polynomial:
 $$f(x,y):=x^{6}+(y^4+y^2-x)^2 (y^2-x)^2.$$  
The shape of its level curves is sketched in Figure \ref{fig:doubleBananaSketch}.

\begin{figure}[H]
\centering
\includegraphics[scale=0.2]{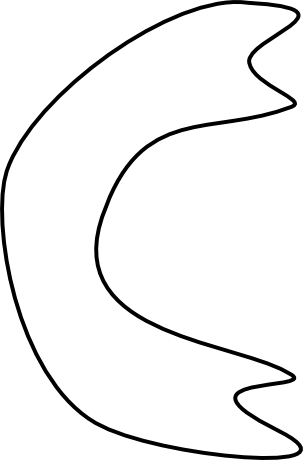} 
\caption{The sketch of a level curve $(f=\varepsilon),$ $f(x,y):=x^{6}+(y^4+y^2-x)^2 (y^2-x)^2$.\label{fig:doubleBananaSketch}}
\end{figure}

\end{example}

\subsection{Star domains}
The following definition is well-known:
\begin{definition}\cite[page 168]{CB}
A set $S$ in the Euclidean space $\bR^n$ is called a \defi{star domain with respect to $x_0\in S$}\index{star domain with respect to a point} if for all $x\in S,$ the line segment from $x_{0}$ to $x$ is in $S$. 
\end{definition}

\begin{proposition}\label{prop:nonStar}
The polynomial function $f:\bR^2\rightarrow\bR$ $$f(x,y):=x^{2p}+(y^2-x)^2$$ for  $p\geq 3$, has the following properties: $O$ is a strict local minimum and for sufficiently small $0<\varepsilon\ll 1$,  
the set $\mathcal{D}_\varepsilon:=f^{-1}([0,\varepsilon])$ is not a star domain with respect to any point   $P\in\mathcal{D}_\varepsilon$.
\end{proposition}

Before the proof, let us first present the following Lemma:
\begin{lemma}\label{lemma:axaSim}
If a set $\mathcal{D}\subset \mathbb{R}^2$ is a star domain with respect to a point and if $\mathcal{D}$ admits a symmetry axis, say $\Delta$ (see Figure \ref{fig:axaSim}), then there exists a point $P\in \mathcal{D}\cap \Delta$ such that $\mathcal{D}$ is a star domain with respect to $P$.
\end{lemma} 

\vspace{-\baselineskip}
\begin{figure}[H]
\centering
\includegraphics[scale=0.2]{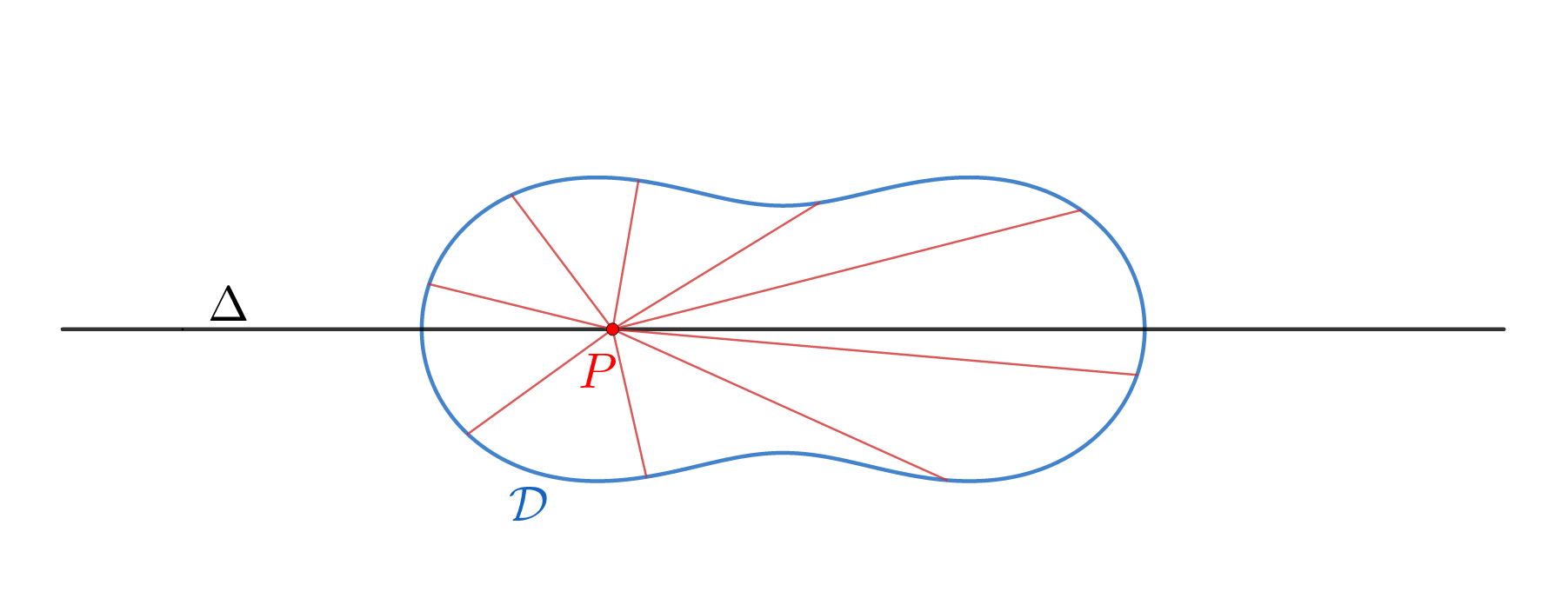} 
\caption{A star domain $\mathcal{D}$ with a symmetry axis $\Delta$.\label{fig:axaSim}}
\end{figure}

\begin{proof}
By hypothesis, the given set $\mathcal{D}\subset \mathbb{R}^2$ is a star domain with respect to a point. Denote this point by $P_1$. If $P_1\in \Delta$, there is nothing to prove. Let us consider $P_1\not\in \Delta$, as pictured in Figure \ref{fig:centruSim2}.

\begin{figure}[H]
\centering
\includegraphics[scale=0.2]{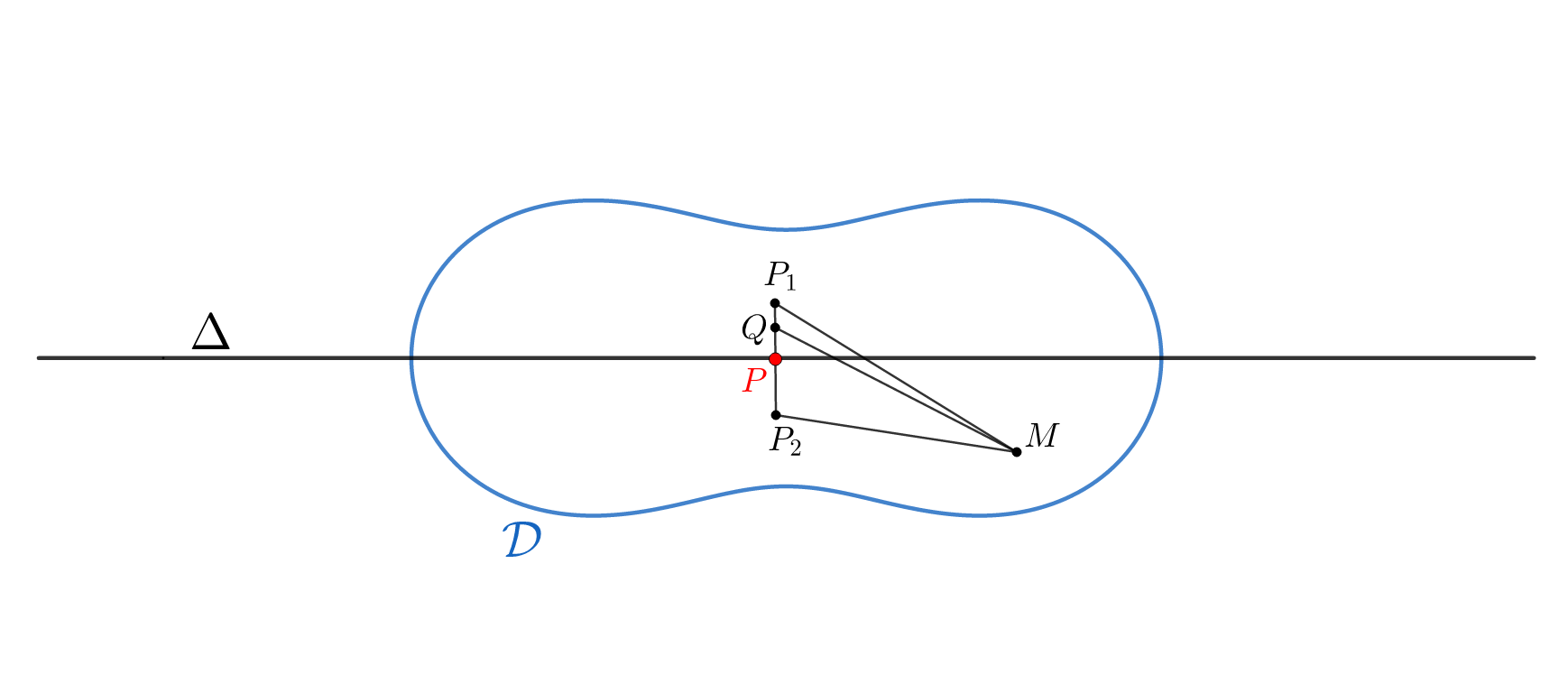} 
\caption{There exists $P\in \mathcal{D}\cap\Delta$ such that $\mathcal{D}$ is a star domain with respect to $P$.\label{fig:centruSim2}}
\end{figure}

By hypothesis, $\Delta$ is a symmetry axis for $\mathcal{D}.$ Let us denote by $P_2\in\mathcal{D}$ the symmetric of $P_1$ with respect to $\Delta.$ Denote by $P$ the midpoint of the segment $[P_1 P_2]$. Hence, $P\in\Delta$ and by symmetry, $\mathcal{D}$ is a star domain also with respect to $P_2.$ Let us now prove that $\mathcal{D}$ is a star domain with respect to any point $Q\in[P_1 P_2],$ hence $\mathcal{D}$ is a star domain with respect to the point $P$.

Let $M\in\mathcal{D}$. Since $\mathcal{D}$ is a star domain with respect to both $P_1$ and $P_2,$ we have that $[MP_1]$, $[MP_2]$ and $[P_1 P_2]$ are included in $\mathcal{D}$. Thus, both the triangle $\bigtriangleup MP_1 P_2$ and its interior are included in $\mathcal{D},$ since the interior is the union of $]P_1 X[$, for all $X\in ]P_2 M[$. In particular, for any point $Q\in[P_1 P_2],$ we obtain $[QM]\subset\mathcal{D}.$ Namely, $\mathcal{D}$ is a star domain with respect to any point $Q\in[P_1 P_2].$ In conclusion, $\mathcal{D}$ is a star domain with respect to $P.$ 
\end{proof}

In the following let us present the proof of Proposition \ref{prop:nonStar}.

\begin{proof}

By Lemma \ref{lemma:axaSim}, it is sufficient to prove that $\mathcal{D}_\varepsilon$ is not a star domain with respect to any point $B(x_B,0)$ on $(y=0)\cap \mathcal{D}_\varepsilon,$ such that $f(B)\leq\varepsilon$, because the line $\Delta:=(y=0)$ is a symmetry axis for $\mathcal{D}_\varepsilon.$

\begin{figure}[H]
\centering
\includegraphics[scale=0.18,trim={0 6cm 0 2cm}, clip]{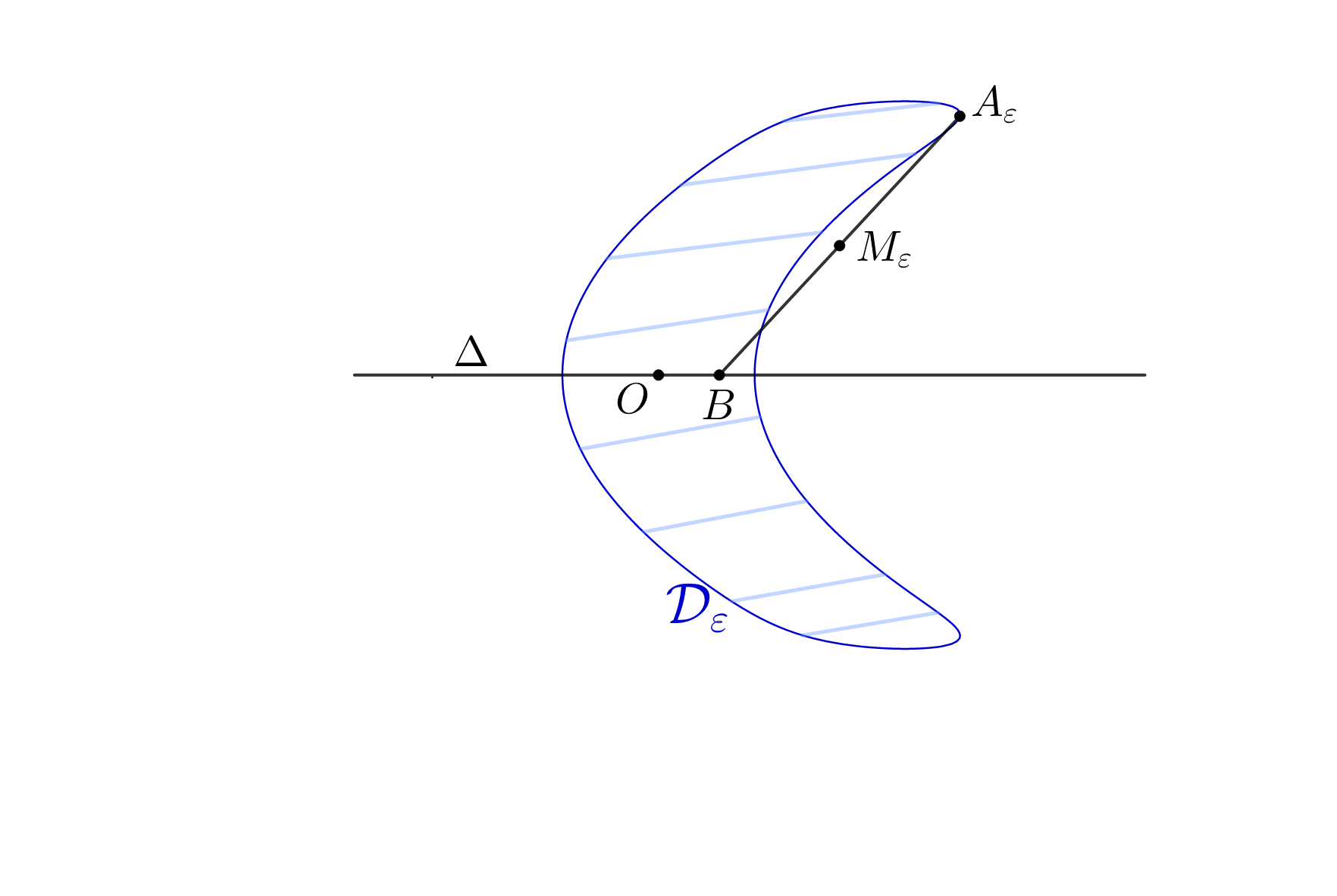} 
\caption{Example of a non-star domain.\label{fig:nonstarDomain}}
\end{figure}

Consider the point $A_\varepsilon\left (\varepsilon^{\frac{1}{2p}},\varepsilon^{\frac{1}{4p}}\right )\in (y^2-x=0)\cap (f=\varepsilon),$ see Figure \ref{fig:nonstarDomain}. Then the midpoint of the segment $[BA_\varepsilon]$ (see Figure \ref{fig:nonstarDomain} below) is the point $$M_\varepsilon\left (\frac{\varepsilon^{\frac{1}{2p}}+x_B}{2},\frac{\varepsilon^{\frac{1}{4p}}}{2}\right ).$$ We will prove that $M_\varepsilon\not \in \mathcal{D}_\varepsilon.$

Since $f(B)\leq\varepsilon,$ we have $x_B^{2p}+x_B^2\leq\varepsilon.$ For sufficiently small $0<\varepsilon\ll 1$, we have $x_B$ sufficiently small, thus $x_B^{2p}<x_B^2.$ Hence $\vert x_B\vert<\sqrt{\varepsilon}.$

We have $f(M_\varepsilon)-\varepsilon>(\frac{1}{4}\varepsilon^{\frac{1}{2p}}-\frac{1}{2}(\varepsilon^{\frac{1}{2p}}+x_B))^2-\varepsilon=\frac{1}{16}\varepsilon^{\frac{1}{p}}+\frac{1}{2}\varepsilon^{\frac{1}{2p}}x_B+\frac{1}{4}x_B^2-\varepsilon$.

There are two cases to consider. 

$\bullet$ First, if $x_B\geq 0$, then we have $f(M_\varepsilon)-\varepsilon \geq\frac{1}{16}\varepsilon^{\frac{1}{p}}-\varepsilon>0$ for a sufficiently small $0<\varepsilon\ll 1.$

$\bullet$ Secondly, if $x_B<0,$ then $f(M_\varepsilon)-\varepsilon\geq \frac{1}{16}\varepsilon^{\frac{1}{p}}+\frac{1}{2}\varepsilon^{\frac{1}{2p}}x_B-\varepsilon.$ Since $\vert x_B \vert \leq \sqrt{\varepsilon}$ and $\varepsilon^{\frac{1}{2p}}<1,$ we obtain $\vert x_B\vert  \varepsilon^{\frac{1}{2p}}\leq \sqrt{\varepsilon}.$ Thus, $f(M_\varepsilon)-\varepsilon\geq \frac{1}{16}\varepsilon^{\frac{1}{p}}-\frac{1}{2}\sqrt{\varepsilon}-\varepsilon>0$ for a sufficiently small $0<\varepsilon\ll 1$ and for $p \geq 3.$ 

In both cases, we get $f(M_\varepsilon)>\varepsilon,$ hence $M\not \in \mathcal{D}_\varepsilon$ and thus the segment $[BA_\varepsilon]$ is not included in $\mathcal{D}_\varepsilon,$ even though $A_\varepsilon, B \in \mathcal{D}_\varepsilon.$ In other words, $\mathcal{D}_\varepsilon$ is not a star domain with respect to any $B\in \mathcal{D}_\varepsilon \cap \Delta.$ 
\end{proof}

\section{The Poincaré-Reeb tree}\label{sect:ConstructionPR}

\subsection{The Poincaré-Reeb construction}\label{subsec:PReebConstr}
 The main purpose of this section is to introduce a new combinatorial object that measures the non-convexity of a smooth and compact connected component of a real algebraic curve in $\mathbb{R}^2$. We will define the Poincaré-Reeb graph, associated to the given curve and to a direction of projection $x$. It is adapted from the classical construction introduced by H. Poincaré (see \cite[1904, Fifth supplement, page 221]{Po}), which was rediscovered by G. Reeb (in \cite{Re}). Both used it as a tool in Morse theory (see \cite{Ma}). Namely, given a Morse function on a closed manifold, they associated it a graph as a quotient of the manifold by the equivalence 
      relation whose classes are the connected components of the levels of the function. We will perform an analogous 
      construction for a special type of manifold with boundary, namely for a topological disk $\mathcal{D}$ bounded by a smooth and compact connected component of a real algebraic curve. We prove that in our setting the Poincaré-Reeb graph is a plane tree.

$\bullet$ Let us consider a smooth and compact connected component of a real algebraic curve in $\mathbb{R}^2$, denoted by $\mathcal{C}$. Let $\mathcal{D}$ denote the topological disk bounded by $\mathcal{C}$ (see Figure \ref{fig:generalReeb} below). Moreover, let us define the projection $\Pi:\mathbb{R}^2\rightarrow\mathbb{R},$ $\Pi(x,y):=x.$
 
\begin{definition}\label{def:equivRel}
For each $x\in\mathbb{R}$, $\Pi^{-1}(x)\cap\mathcal{D}$ is a finite union of connected vertical segments. Let us define the \defi{equivalence relation $\sim$}\index{equivalence relation $\sim$} as follows: if $(x,y)$ and $(x,y')$ are two points in the real plane $\mathbb{R}^2$ then $(x,y)\sim (x,y')$ if and only if they  belong to the same vertical connected component of the fibre $\Pi^{-1}(x)\cap\mathcal{D}$.
\end{definition}

Given the above equivalence relation, let us consider the canonical quotient map $$q:\mathbb{R}^2\rightarrow\mathbb{R}^2/\sim.$$

\begin{figure}[H]
\centering
\includegraphics[scale=0.2]{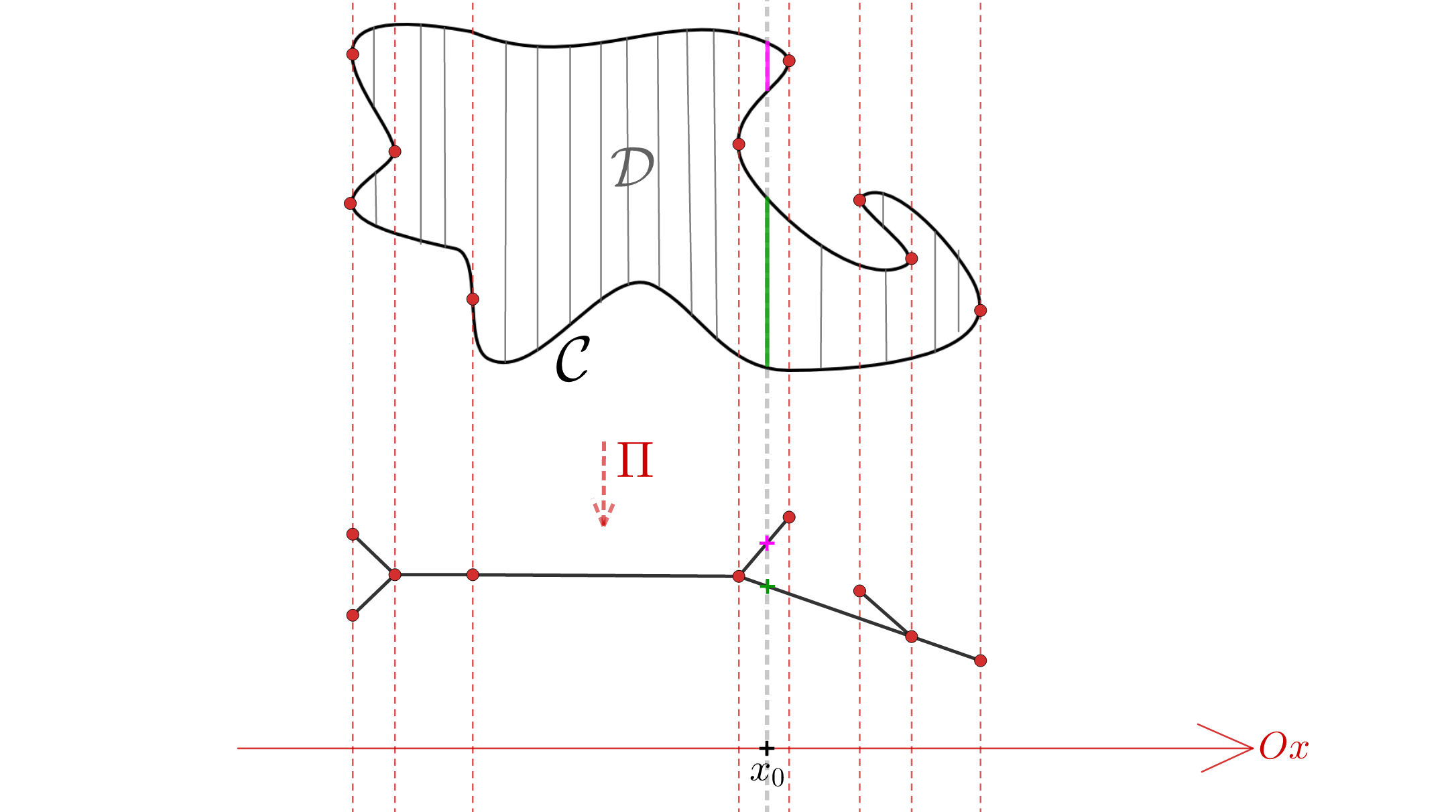} 
\caption{A Poincaré-Reeb graph of a smooth and compact connected component of a real algebraic curve $\mathcal{C}$.\label{fig:generalReeb}}
\end{figure}

Denote by $P:=\mathbb{R}^2$ and by $\tilde{P}:=\mathbb{R}^2/\sim.$ The coordinates in $\tilde{P}$ are $(\tilde{x},\tilde{y}),$ with $\tilde{x}=x.$ 

The function $\Pi$ descends to the quotient in a function $\tilde{\Pi}:\tilde{P}\rightarrow \mathbb{R},$ $\tilde{\Pi}(x,\tilde{y}):=x.$ The function $\tilde{\Pi}$ is continuous, since $\Pi$ is continuous.

We obtain the following commutative diagram (see Figure \ref{fig:commDiagr}):

\begin{figure}[H]
\centering
\includegraphics[scale=0.22,trim=0 350 0 240,clip]{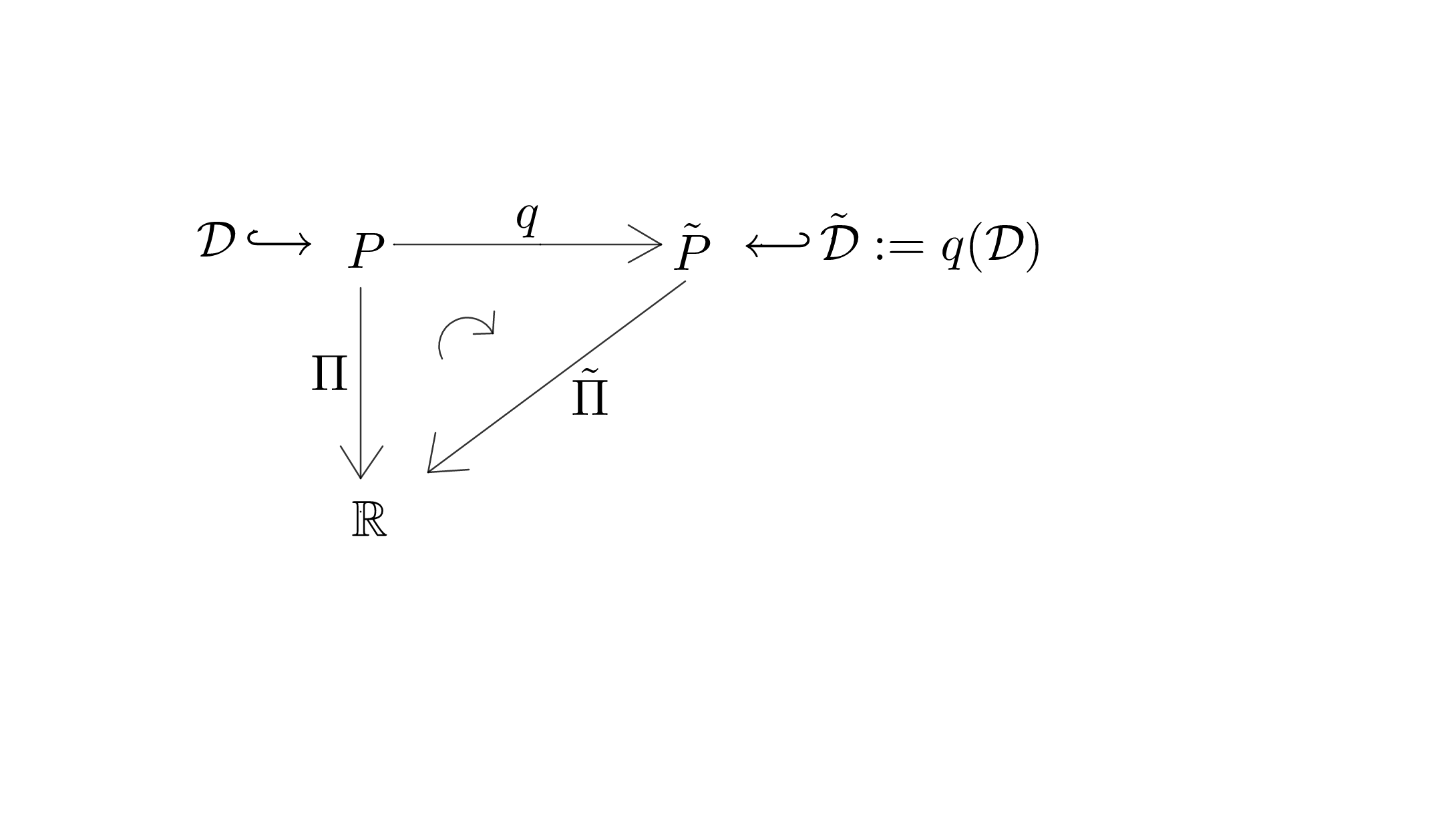} 
\caption{The diagram is commutative: $\Pi=\tilde{\Pi}\circ q$.\label{fig:commDiagr}}
\end{figure}

$\bullet$ Our next aim is to prove that $\tilde{P}$ is homeomorphic to $\mathbb{R}^2$. First we show that $\tilde{P},$ endowed with the function $\tilde{\Pi}$ is a fibre bundle, i.e. that satisfies a local triviality. We will prove that it is locally a cartesian product of two sub-spaces. Furthermore, we shall prove that the local trivialisation of the fibre bundle $\tilde{P}$ implies that $\tilde{P}$ is a trivial fibre bundle, i.e. it is not just locally a product of two spaces, but globally.

\begin{proposition}
The topological space $\tilde{P}$ endowed with the function $\tilde{\Pi}$ is a fibre bundle.
\end{proposition}

\begin{proof}
By the definition of the fibre bundle (see for instance \cite[page 51]{Fr}), what we need to prove is that there exists a local trivialisation of $\tilde{P}$. Let us consider a point $x_0\in\mathbb{R}$ in the base space, that is the $Ox$-axis. Take an open neighbourhood of $x_0,$ called $U\subset\mathbb{R}.$ We want to show that the set $\tilde{\Pi}^{-1}(U)$ is homeomorphic to $U\times\textbf{R},$ here by $\textbf{R}$ we mean a space that is homeomorphic to the real line $\mathbb{R}.$ To this end, we have the following two arguments:

1. If on a vertical line we contract a finite number of segments into points (by the quotient map $q$ described above), we obtain a topological space that is homeomorphic to $\mathbb{R}.$ Denote this space by $\textbf{R}.$

2. The set $\mathcal{D}\cap\Pi^{-1}(x_0)$ is a finite union of intervals (some of them empty intervals), whose extremities depend continuously on $x$ (see Figure \ref{fig:extremitiesContinX}), except when $x_0$ is a critical point.

Hence we have the local triviality: $\tilde{\Pi}^{-1}(U)$ is homeomorphic to $U\times \textbf{R}.$ Thus the image $\tilde{P}$ of the plane $P$, by the quotient map $q$, is a fibre bundle.

\begin{figure}[H]
\centering
\includegraphics[scale=0.2]{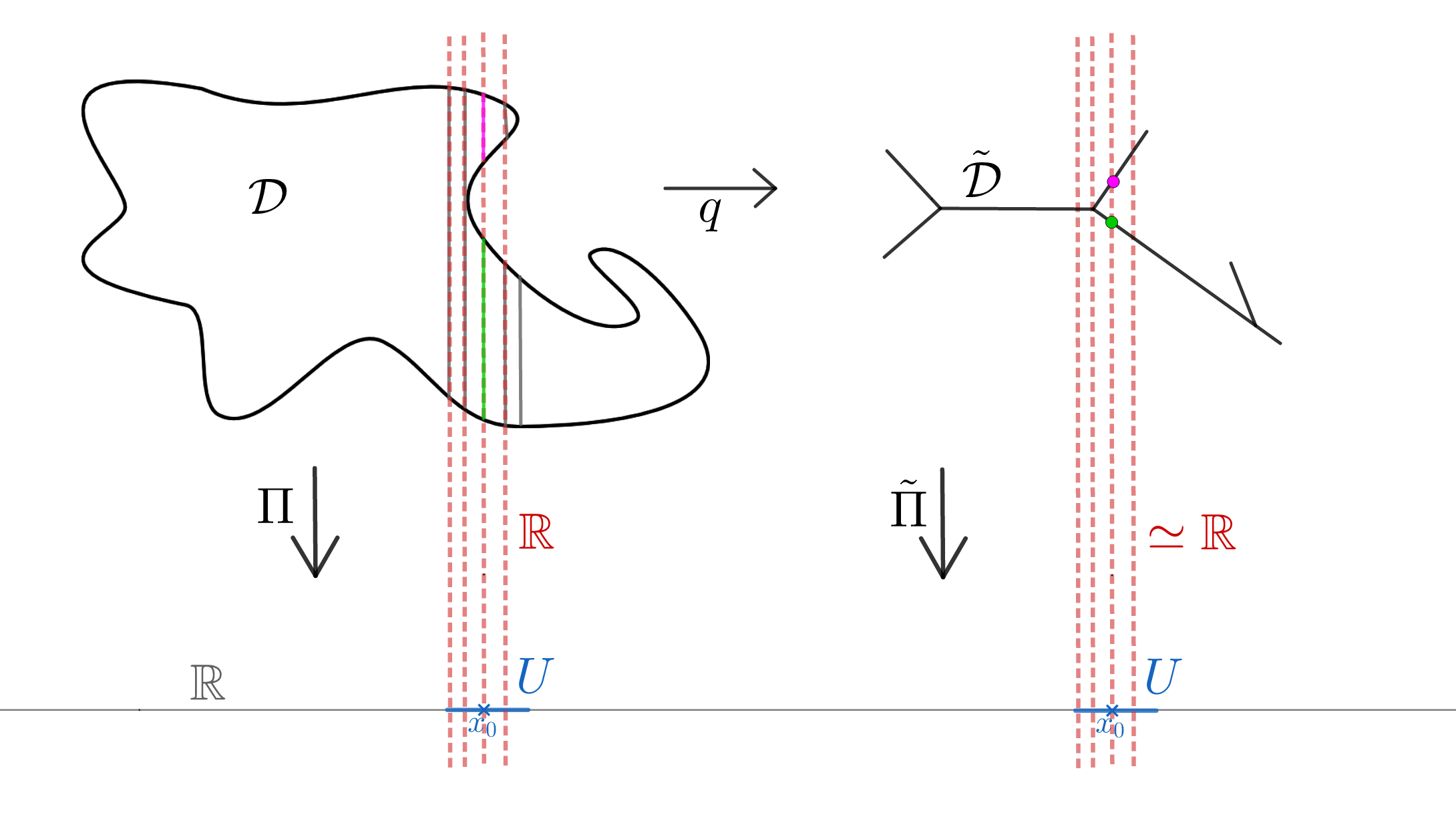} 
\caption{The extremities of the intervals in $\mathcal{D}\cap\Pi^{-1}(x_0)$ depend continously on $x$. After the quotient, we obtain vertical lines homeomorphic to $\mathbb{R}.$ \label{fig:extremitiesContinX}}
\end{figure}

\end{proof}

\begin{proposition}
The topological space $\tilde{P}$ endowed with the function $\tilde{\Pi}$ is a \textbf{trivial} fibre bundle.
\end{proposition}

\begin{proof}
The base space of the fibre bundle $\tilde{P}$ is the $Ox$-axis, namely $\mathbb{R}$, and $\mathbb{R}$ is contractible to a point. Thus we conclude that $\tilde{P}$ is a trivial fibre bundle, since by \cite[Proposition 3.5, page 75]{Os} a fibre bundle over a base space that is contractible to a point is trivial. To be more precise, we have $\tilde{P}$ homeomorphic to $\mathbb{R}^2.$
\end{proof}

\begin{corollary}
The topological subspace $\tilde{\mathcal{D}}:=q(\mathcal{D})$ is embedded in a real plane.
\end{corollary}

\begin{theorem}
The topological subspace $\tilde{\mathcal{D}}$ is connected.
\end{theorem}

\begin{proof}
The image of a connected space by a continuous function is connected. Since $\tilde{\mathcal{D}}:=q(\mathcal{D})$, the topological disk $\mathcal{D}$ is connected and the quotient map $q$ is continuous, the proof is complete.
\end{proof}

\begin{theorem}\label{th:transverse}
The topological subspace $\tilde{\mathcal{D}}$ is a one-dimensional subspace of $\tilde{P}$. In each of the regular points of $\tilde{\mathcal{D}}$, we have $\tilde{\mathcal{D}}$ transverse to the foliation of the plane $\tilde{P}$ by vertical lines, induced by $\tilde{\Pi}$.
\end{theorem}

\begin{proof}
Since the curve $\mathcal{C}$ is a connected component of a real algebraic curve, that is, it is an analytic curve, it has finitely many points of vertical tangency. In other words, there are finitely many critical levels of the projection $\Pi_{|\mathcal{C}}:\mathcal{C}\rightarrow\mathbb{R}$. Locally near a critical point, one may have a tangency point of even order (see Figure \ref{fig:evenTg}) or of odd order (see Figure \ref{fig:oddTg}):

\vspace{-\baselineskip}
\begin{figure}[H]
\centering
\begin{subfigure}[b]{0.42\textwidth}
\includegraphics[scale=0.13,trim={0 0 0 3cm}, clip]{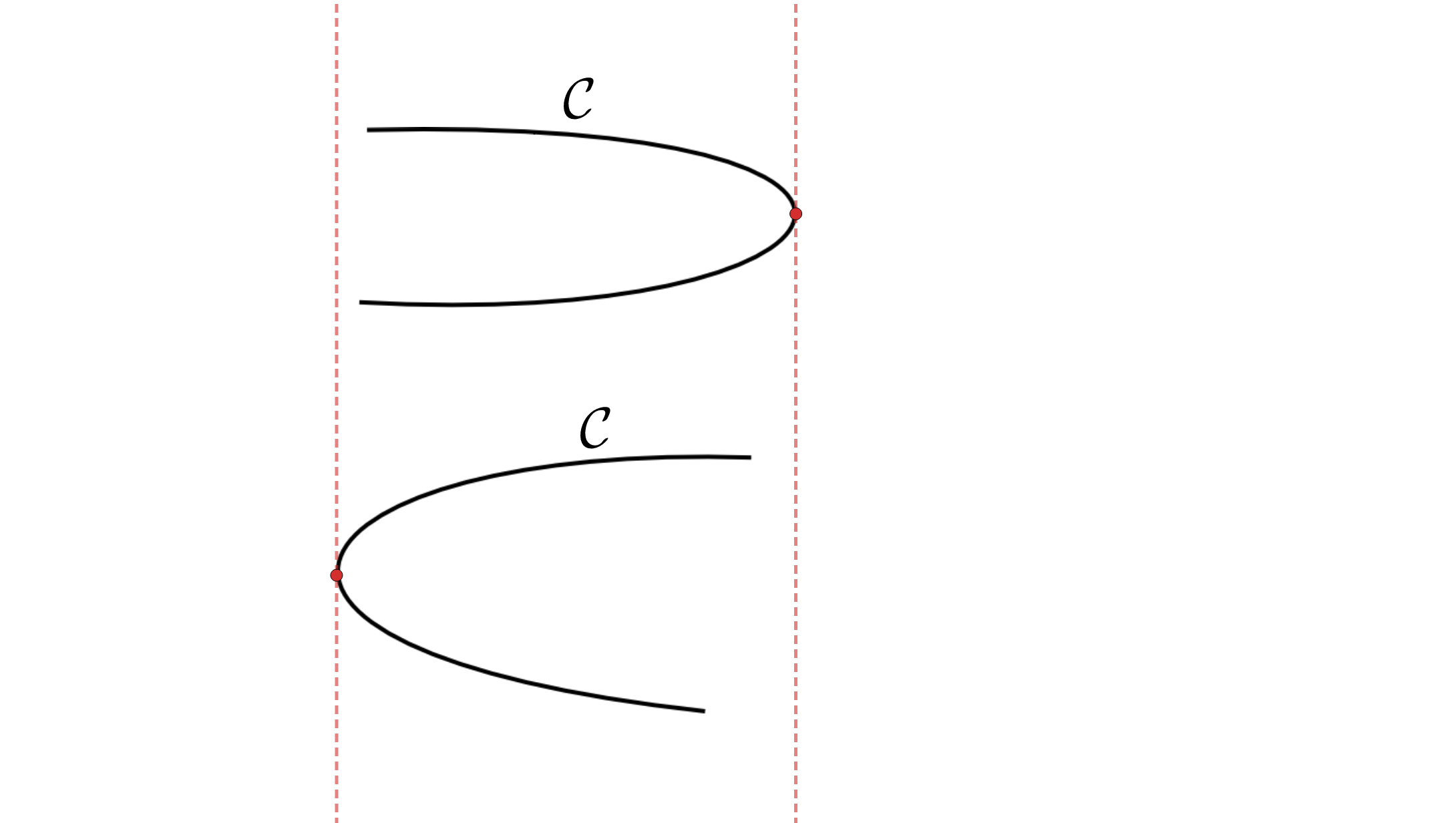} 
\caption{Points on $\mathcal{C}$ of vertical even tangency\label{fig:evenTg} }
\end{subfigure}
\begin{subfigure}[b]{0.42\textwidth}
\includegraphics[scale=0.13,trim={0 0 0 3cm}, clip]{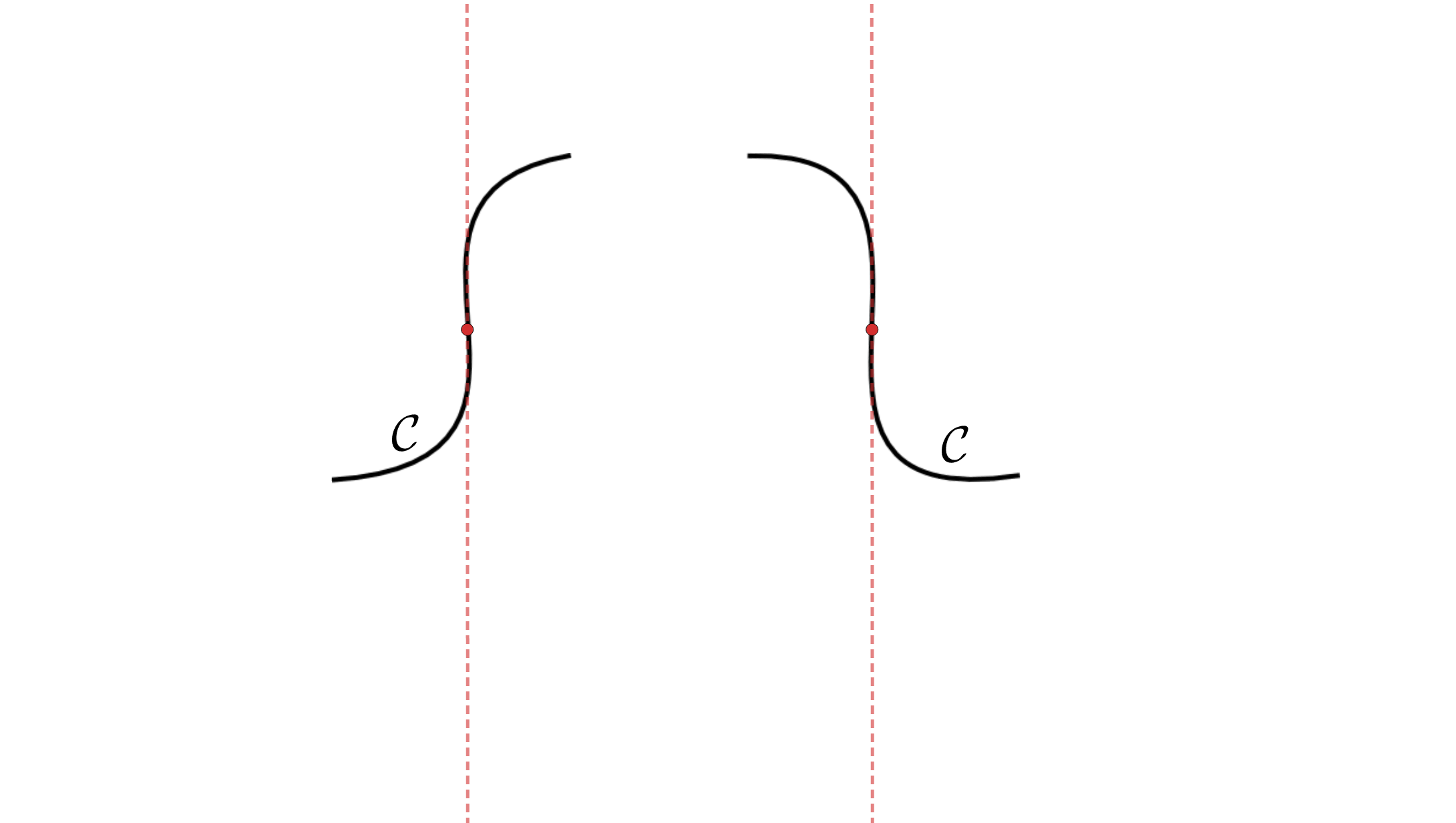} 
\caption{Points on $\mathcal{C}$ of vertical odd tangency\label{fig:oddTg} }
\end{subfigure}
\caption{Vertical tangents.}
\end{figure}

Between any two critical levels, we have only what we call \enquote{bands} of the initial disk $\mathcal{D}$. See, for a better comprehension, Figure \ref{fig:bands} below. After taking the quotient, each band contracts into an arc of dimension one, transverse to the vertical foliation of the real plane $\tilde{P}$. Therefore, the topological space $\tilde{\mathcal{D}}$ is a graph, being a finite union of transverse arcs.

\begin{definition}\cite[pages 13-14]{Vic}
A \defi{preorder}\index{preorder} on a set is a binary relation that is reflexive and transitive.
\end{definition}

\begin{definition}\label{DefReeb}
Given the projection $\Pi:\mathbb{R}^2\rightarrow\mathbb{R},$ let us consider the image $\tilde{\mathcal{D}}:=q(\mathcal{D}),$ of the topological disk $\mathcal{D}$ bounded by the smooth and compact connected component of a real algebraic curve $\mathcal{C}$, where $q$ the quotient map introduced before. The image $\tilde{\mathcal{D}},$ endowed with special vertices, namely the finitely many points corresponding to the critical points of $\Pi$ on $\mathcal{C}$ is called \defi{the Poincaré-Reeb graph}\index{Poincaré-Reeb graph} associated to the curve $\mathcal{C}$ and to $\Pi$. Endow the real plane with the trivial fibre bundle $\tilde{P}$, then endow the vertices of the graph with the induced preorder.
\end{definition}
\begin{figure}[H]
\centering

\includegraphics[scale=0.25,trim={0 5cm 0 4cm}, clip]{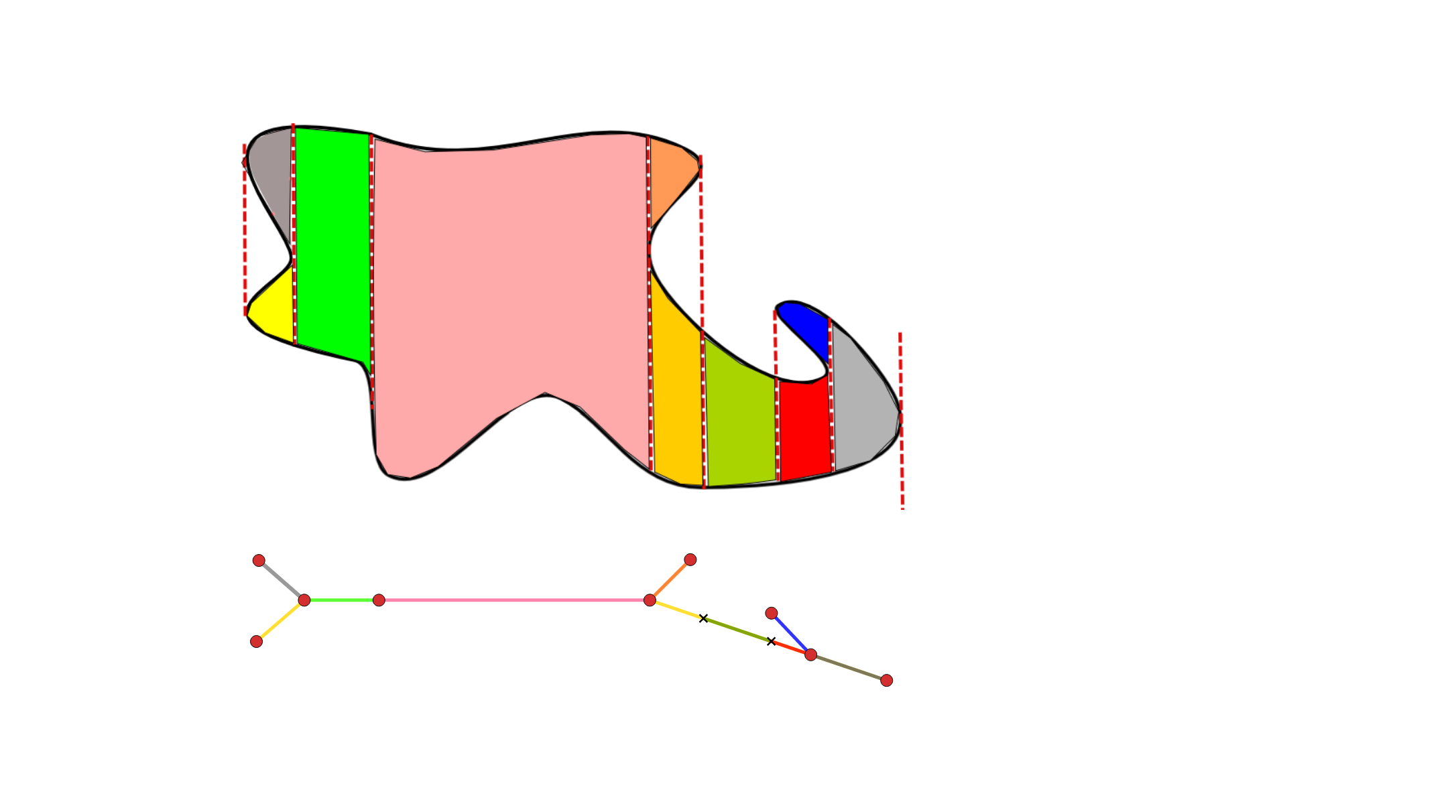} 
\caption{\enquote{Bands} of the disk $\mathcal{D}$, between two consecutive critical levels of the vertical projection $\Pi$. The bands contract into one-dimensional transverse arcs.\label{fig:bands}}
\end{figure}

\end{proof}

In conclusion, the Poincaré-Reeb construction that we introduce in this paper gives us a plane connected graph, embedded in a surface that is homeomorphic to $\mathbb{R}^2$. We will prove in the next section that it is in fact a tree.

\subsection{Constructible functions and a Fubini type theorem}\label{Sect:TreeFubini}

In this section we will recall the integration of constructible functions with respect to the Euler characteristic and a Fubini-type theorem for it. Then we will apply it to the proof of Proposition \ref{Prop:ReebOfDiskIsTree}.

We will first introduce the necessary tools: constructible functions. We shall follow the definitions and results from \cite[Chapter 3, pages 21-22]{Co}. For more details, the reader should refer to \cite[page 162]{Wa}. Constructible functions have applications in the study of the topology of singular real algebraic sets (see for instance \cite{ND}). The main tool is the integration against the Euler characteristic, which was first described in Viro's paper \cite{Vi}.

\begin{definition}\cite[page 25]{Co1}
A \defi{basic semialgebraic set}\index{semialgebraic!basic semialgebraic set} is a subset of $\mathbb{R}^n$ satisfying a finite number of polynomial equations and inequalities with real coefficients.
\end{definition}

\begin{example}
\ \\
$\bullet$ All algebraic sets are basic semialgebraic.
\ \\
$\bullet$ Unions of finitely many points and open intervals in $\mathbb{R}$ are basic semialgebraic sets.
\end{example}

\begin{definition}\cite[page 29]{Co1}
Let us consider two basic semialgebraic sets $M\subset\mathbb{R}^m$ and $N\subset\mathbb{R}^n$. We say that a continuous map $s:M\rightarrow N$ is a \defi{semialgebraic map}\index{semialgebraic!map} if its graph
$\{(x,y)\in M\times N\mid y=s(x)\}$
is a semialgebraic subset of $\mathbb{R}^m\times\mathbb{R}^n$.
\end{definition}

\begin{definition}\cite[Chapter 3, page 21]{Co}
A \defi{constructible function}\index{function!constructible function} on a semialgebraic set $A$ is a function $\varphi:A\rightarrow\mathbb{Z}$ which takes finitely many values and such that, for every $n\in\mathbb{Z}$, $\varphi^{-1}(n)$ is a semialgebraic subset of $A$.
\end{definition}

\begin{remark}\cite[page 162]{Wa}
Each constructible function is an integer linear combination of characteristic functions of constructible sets $X_i$, i.e. of semialgebraic subsets of $A$:
$$\varphi=\sum_{i\in I}m_i\mathbf{1}_{X_i},$$
where $m_i\in\mathbb{Z}$ and $\mathbf{1}_{X_i}$ is the characteristic function of $X_i$.
\end{remark}

\begin{definition}\cite[Chapter 3, page 22]{Co}
Let $\varphi$ be a constructible function on a semialgebraic set $X$. The \defi{integral} of $\varphi$ \textbf{with respect to the Euler characteristic $\chi$}\index{integral with respect to the Euler characteristic $\chi$} is by definition:
$$\int_{X}\varphi\mathrm{d}\chi:=
\sum_{n\in\mathbb{Z}}n\chi(\varphi^{-1}(n)\cap X),$$
where $\chi$ denotes the Euler characteristic.
\end{definition}

\begin{definition}\cite[Chapter 3, page 22]{Co}\label{def:Pushforward}
Let us consider a continuous semialgebraic map $f:A\rightarrow B$ and let $\varphi$ be a constructible function on $A$. The \defi{pushforward}\index{pushforward} of $\varphi$ \textbf{along} $f$ (see Figure \ref{fig:pushfor}) is the function $f_*\varphi:B\rightarrow \mathbb{Z}$, $$f_*\varphi(y):=\int_{f^{-1}(y)}\varphi\mathrm{d}\chi.$$
\end{definition}

\begin{figure}[H]
\begin{center}
\includegraphics[scale=0.12,trim={0 6cm 0 7cm}, clip]{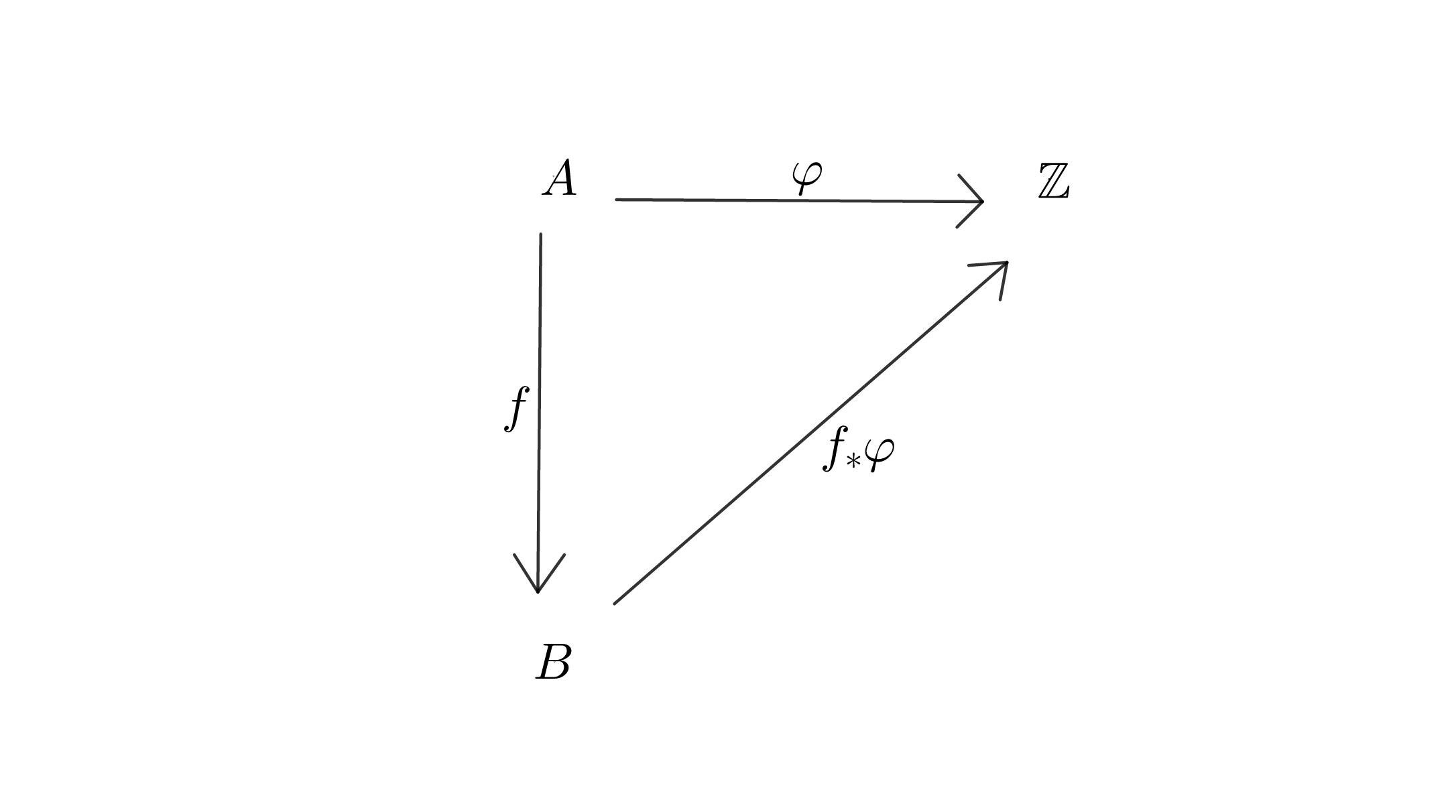} 
\end{center}\caption{The pushforward $f_*\varphi$ of $\varphi$ along $f$.\label{fig:pushfor}}
\end{figure}

\begin{remark}
If we apply the definition of push-forward to the application from a set $A$ towards a point, we obtain the function that associates at this point the integral with respect to the Euler characteristic of the starting function on $A$.
\end{remark}

\begin{remark}
The diagram from Figure \ref{fig:pushfor} is not commutative. If, for example, $A:=\{a,b\}$, $B:=\{c\}$, where $a,b,c$ are points in $\mathbb{R}^2$, $f:A\rightarrow B$ a semialgebraic map such that $f(a)=f(b):=c$ and $\varphi$ a constructible function on $A$ such that $\varphi(a):=-1$, $\varphi(b):=1$, then by definition we compute $f_*\varphi(b)=\int_{f^{-1}(c)}\varphi\mathrm{d}\chi=-1+1=0.$ Thus $f_*\varphi \circ f(a)=0$ while $\varphi(a)=-1$. Therefore $f_*\varphi \circ f\neq\varphi.$
\end{remark}

\begin{remark}
The following proofs can be extended to semianalytic sets, with the more general hypothesis that $f$ is analytic.
\end{remark}

\begin{theorem}\cite[Chapter 3, page 22]{Co}\label{Th:Fubini}
\textbf{Fubini Type Theorem:} Let $f:A\rightarrow B$ be a semialgebraic map and $\varphi$ a constructible function on $A$. Then 
$$\int_{B}f_*\varphi\mathrm{d}\chi=\int_{A}\varphi \mathrm{d}\chi.$$
\end{theorem}

We are ready now to prove the main result of this section, namely Proposition \ref{Prop:ReebOfDiskIsTree}.
\begin{proposition}\label{Prop:ReebOfDiskIsTree}
The Poincaré-Reeb graph of a topological disk $\mathcal{D}$ bounded by a smooth and compact connected component of a real algebraic curve in $\mathbb{R}^2$ is a tree. 
\end{proposition}

\begin{proof}
In Definition \ref{def:Pushforward} let us replace $A$ with the disk $\mathcal{D},$ $B$ with the graph $G$ obtained after the Poincaré-Reeb contraction $\pi$, the function $f$ with the Poincaré-Reeb contraction $\pi$ and $\varphi$ with the characteristic function $\mathbf{1}_{\mathcal{D}}$, as in Figure \ref{fig:pushforEx}:

\begin{figure}[H]
\begin{center}

\includegraphics[scale=0.12,trim={0 6cm 0 7cm}, clip]{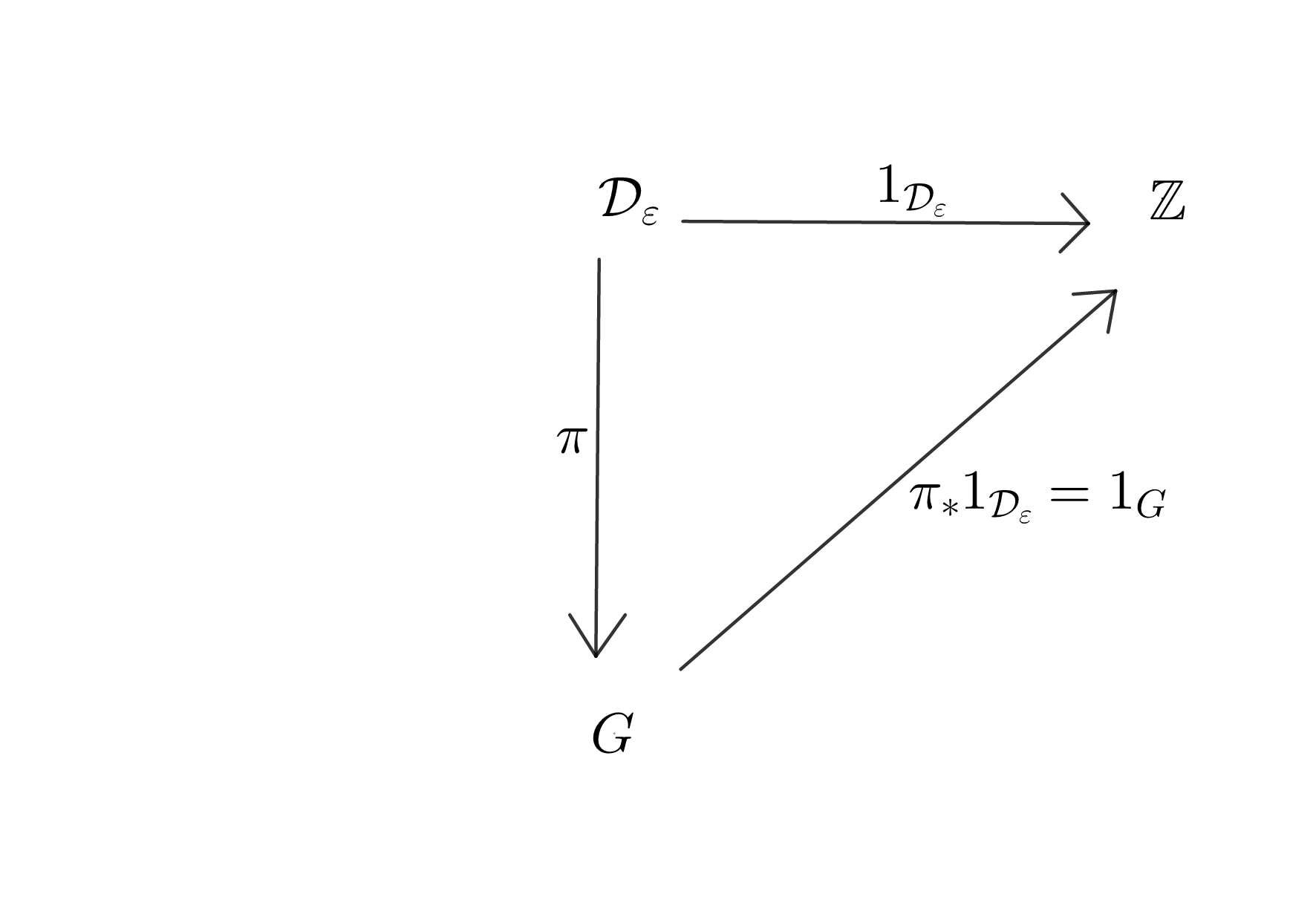} 
\end{center}\caption{The pushforward $\pi_*\mathbf{1}_{\mathcal{D}}$ of $\mathbf{1}_{\mathcal{D}}$ along $\pi$.\label{fig:pushforEx}}
\end{figure}

In other words, we apply the Fubini Type Theorem \ref{Th:Fubini} for the pushforward of the constructible function $\mathbf{1}_{\mathcal{D}}$.

Now for $g\in G$, we obtain: $$\pi_*\mathbf{1}_{\mathcal{D}}(g)=\int_{\pi^{-1}(g)}\mathbf{1}_{\mathcal{D}}\mathrm{d}
\chi=\sum_{n\in\mathbb{Z}}n\chi(\mathbf{1}_{\mathcal{D}}^{-1}(n)\cap \pi^{-1}(g))
=\chi(\pi^{-1}(g))=1,$$ since $\pi^{-1}(g)$ is a connected component of the fibre, i.e. a segment. Thus $$\pi_*\mathbf{1}_{\mathcal{D}}=\mathbf{1}_G.$$

By Theorem \ref{Th:Fubini}, we have now $$\int_{\mathcal{D}}\mathbf{1}_{\mathcal{D}}\mathrm{d}\chi=\int_{G}\pi_*\mathbf{1}_D\mathrm{d}\chi.$$ Since $\pi_*\mathbf{1}_{\mathcal{D}}=\mathbf{1}_G,$ we obtain: $$\int_{\mathcal{D}}\mathbf{1}_{\mathcal{D}}\mathrm{d}\chi=\int_{G}\mathbf{1}_G\mathrm{d}\chi.$$ Therefore, $\chi({\mathcal{D}})=\chi(G).$ Since $\mathcal{D}$ is a disk, we have $\chi({\mathcal{D}})=1.$ In Section \ref{sect:ConstructionPR} we proved that  $G$ is a connected graph, thus the Betti number $b_0=1$. We have $\chi(G)=b_0-b_1.$ Since we proved that its Euler characteristic is $\chi(G)=1,$ we get that the Betti number $b_1=0$. Thus we conclude that $G$ has no cycles, hence $G$ is a tree.

\end{proof}

 \begin{corollary}\label{cor:PReebInGeneral}
 The Poincaré-Reeb graph associated to 
a compact and smooth connected component of a real 
       algebraic plane curve and to
a direction of projection $x$ is a plane tree whose open edges are transverse to the foliation induced by the function $x$. Its vertices are endowed with a total preorder relation induced by the function $x$.
 \end{corollary}

\begin{definition}\label{def:transversalTree}
\ \\
$\bullet$ A \defi{transversal tree}\index{tree!transversal tree} is a finite tree embedded in a real oriented plane $P$ endowed with  a trivialisable fibre bundle $\Pi:P\to\mathbb{R}$, such that 
       each one of its edges is transversal to the fibres of $\Pi$.
  \ \\     
       $\bullet$ Two transversal trees $(P_1,T_1)$ and $(P_2, T_2)$ 
       are considered to be \defi{equivalent}\index{tree!transversal tree!equivalent} if there exists an orientation preserving homeomorphism $\psi$ sending one pair 
       into the other one and one fibre bundle  into the other one (in the sense that there exists an orientation-preserving 
       homeomorphism $\lambda : \mathbb{R} \to \mathbb{R}$ such that $\Pi_2 \circ \psi = \lambda \circ \Pi_1$).

\end{definition}

\section{Asymptotic behaviour of curves near a strict local minimum}\label{sect:asymp}

Let us consider a polynomial function $f:\mathbb{R}^2\rightarrow\mathbb{R}$ with a strict local minimum at the origin (see Section \ref{SectHypothesis}). The aim of this section is to study the asymptotic shape of the Poincaré-Reeb trees of sufficiently small levels. We will prove that their shapes stabilise to a limit shape which we shall call the \enquote{asymptotic Poincaré-Reeb tree of the strict minimum with respect to the function $x$}. 

\subsection{Nested smooth Jordan curves}\label{sect:prelim}

The following definition and theorem are well-known:
\begin{definition}\cite[page 19]{Kr}
A \defi{Jordan curve}\index{Jordan curve} is the image of a continuous application $\gamma:[a,b]\rightarrow\mathbb{R}^2$, where $a,b\in \mathbb{R}$ are distinct real numbers, such that $\gamma(a)=\gamma(b)$ and $\gamma_{|[a,b[}$ is injective.
\end{definition}

\begin{theorem}\label{th:Jordan}
The complement in $\bR^2$ of a Jordan curve $J$ consists of two connected components, each of which has $J$ as its boundary (see Figure \ref{fig:jordan}). Both components are path-connected and open and exactly one is unbounded. 
\end{theorem}

Proofs of Theorem \ref{th:Jordan} can be found in \cite[pages 119, 133]{Wa1}.

\begin{figure}[H]
\centering
\includegraphics[scale=0.17]{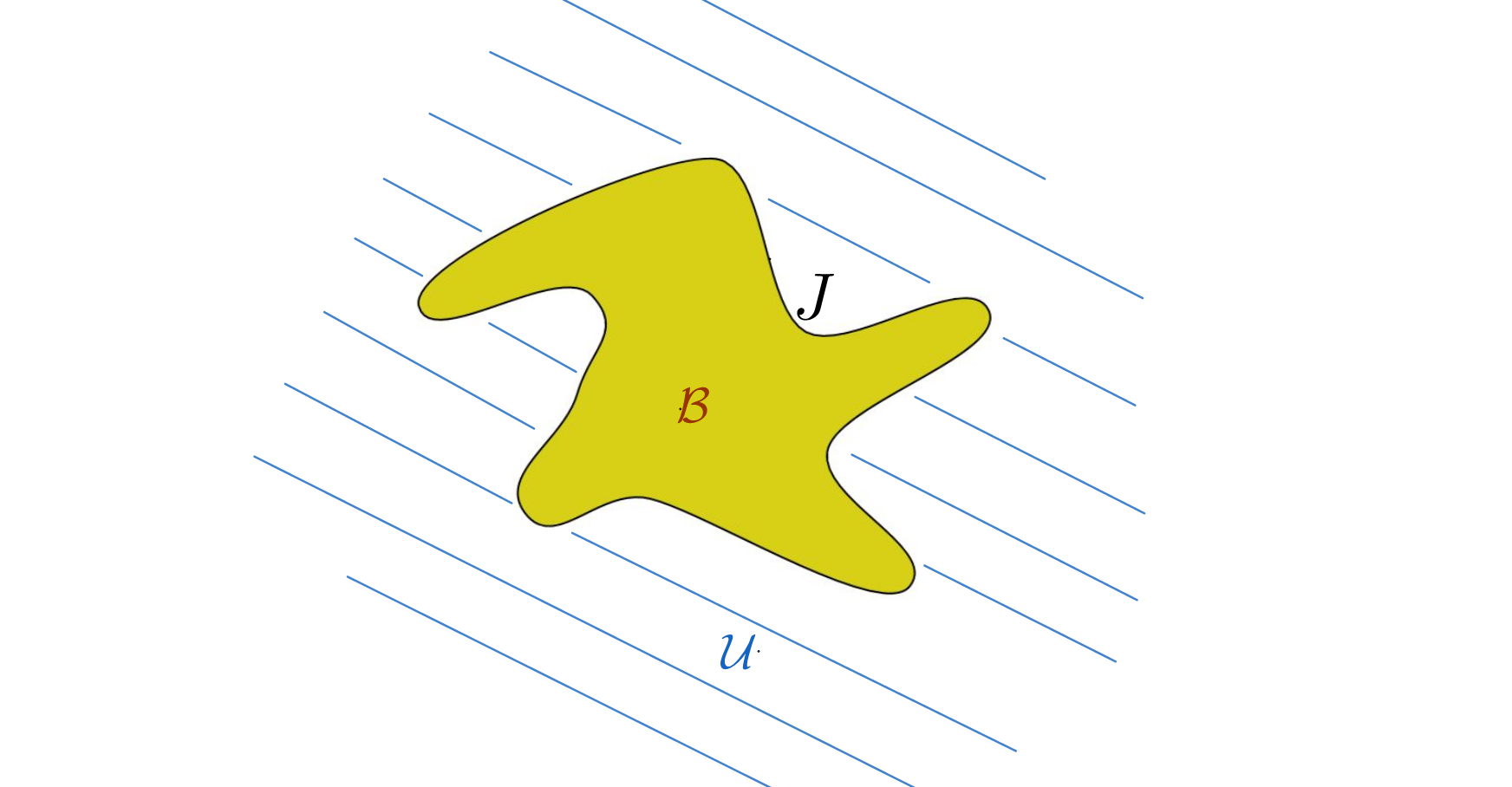} 
\caption{A Jordan Curve $J$: the bounded component of the real plane is $\mathcal{B}$, while the unbounded component is $\mathcal{U}$.\label{fig:jordan}}
\end{figure}

\begin{lemma}\label{lemma:origInInterior}
Under the notations and hypotheses listed in Section \ref{SectHypothesis}, for all sufficiently small $\varepsilon > 0$, the curve $\mathcal{C}_\varepsilon$ is a Jordan curve; more precisely, $\mathcal{C}_\varepsilon$ is diffeomorphic to $S^1$ and $\mathcal{D}_\varepsilon$ is diffeomorphic to a disk. If one denotes by $\Int \mathcal{C}_\varepsilon$ the bounded connected component of $\bR^2\setminus \mathcal{C}_\varepsilon$, then $(0,0)\in\Int \mathcal{C}_\varepsilon.$

\end{lemma}

\begin{proof}
Let us start by giving a slightly different definition of  $\mathcal{C}_\varepsilon$. At the end of this proof, we will have justified Definition \ref{def:levelCEpsilon}.

Suppose that $\mathcal{C}_\varepsilon$ is the union of connected components of $f^{-1}(\varepsilon),$ near the origin.

\textbf{First step: let us prove that $\mathcal{C}_\varepsilon$ is a one-dimensional manifold.}

By hypothesis (Section \ref{SectHypothesis}), $\mathcal{C}_{0}=\{(0,0)\}$ is an isolated local minimum of $f$, thus for a sufficiently small $\varepsilon$, the level curve $\mathcal{C}_\varepsilon$ has no critical points. Therefore, $\mathcal{C}_\varepsilon$ is a one-dimensional manifold. By \cite[Appendix \enquote{Classifying 1-manifolds}, page 65]{Mi1}, the manifold $\mathcal{C}_\varepsilon$ is diffeomorphic to a disjoint union of circles $S^1$, or of some intervals of real numbers: the real line $\bR$, the half-line $\bR_{+}$ or the closed interval $[0,1]$.

\textbf{Second step: we prove that
one cannot obtain components of $\mathcal{C}_\varepsilon$ which are diffeomorphic to  closed or semiclosed segments in $\mathbb{R}$.}

Let us show that $\mathcal{C}_\varepsilon$ has no connected components diffeomorphic to $[0,1]$ or $[0,+\infty[$. The reader should recall that $0<\varepsilon\ll 1$ is a regular value of $f$, thus (see \cite[Preimage Theorem, page 21]{GP}) the set $\mathcal{C}_\varepsilon=f^{-1}(\varepsilon)$ is a submanifold of $\bR^2$ of dimension $1$. 
Let $P\in \mathcal{C}_\varepsilon$ be a point. If $B_{P}$ is an open disk centered at $P$, then
$\mathcal{C}_\varepsilon\cap B_{P}$ is diffeomorphic to an open interval $]a,b[$, where the image of $P$ is in $]a,b[$. Since the point $P$ was chosen arbitrarily, we conclude that one cannot obtain connected components of $\mathcal{C}_\varepsilon$ diffeomorphic to either closed or semiclosed segments in $\mathbb{R}$. 

\textbf{Third step: 
let us prove that $\mathcal{C}_\varepsilon$ is a disjoint union of connected components which are all diffeomorphic to the circle $S^1.$}

We fix $r\in\bR$ small enough, $r>0$ and we fix the neighbourhood $V$ to be $D_r:=\{(x,y)\in\mathbb{R}^2\mid \vert \vert (x,y)\vert \vert\leq r\}$ such that $(0,0)$ is the only critical point of $f$ in the disk $D_{r}.$
Firstly, we have $\mathcal{C}_\varepsilon$ is a closed set, since the preimage of a closed set by the continuous function $f$ is closed. Since $\mathcal{C}_\varepsilon$ is also bounded, it is compact.

Secondly, we will show that $\mathcal{C}_\varepsilon\cap\partial D_r=\emptyset$. Since for $\varepsilon$ sufficiently small, the curve $(f=\varepsilon)$ lies sufficiently close to the origin $(0,0),$ there exists $0<\varepsilon_1\ll 1$ such that $\mathcal{C}_\varepsilon\cap D_{\frac{r}{2}}\neq \emptyset$ for all $\varepsilon<\varepsilon_1.$ We want to prove that there exists $0<\varepsilon_2\ll \varepsilon_1$ such that $\mathcal{C}_\varepsilon\subset D_{\frac{r}{2}}$ for all $\varepsilon<\varepsilon_2.$ We argue by contradiction. Suppose there exists a sequence $\varepsilon_n\rightarrow 0$ such that $\mathcal{C}_{\varepsilon_n}\not\subset D_{\frac{r}{2}},$ namely $\mathcal{C}_{\varepsilon_n}\cap D_r\setminus D_{\frac{r}{2}}\neq\emptyset.$ Hence, for any $\varepsilon_n,$ there exists a point $(x_{\varepsilon_n},y_{\varepsilon_n})\in\mathcal{C}_{\varepsilon_n}$, namely $f(x_{\varepsilon_n},y_{\varepsilon_n})=\varepsilon_n$, such that $(x_{\varepsilon_n},y_{\varepsilon_n})\in {D_{r}}\setminus D_{\frac{r}{2}}.$ If $\varepsilon:=\frac{1}{n}$, then one obtains a sequence $(x_n,y_n)$ with the following properties:
\begin{enumerate}
\item $(x_n,y_n)\in D_{r}\setminus D_{\frac{r}{2}};$
\item $f(x_n,y_n)=\frac{1}{n}$.
\end{enumerate}
Since $D_{r}\setminus D_{\frac{r}{2}}$ is a compact, there exists a convergent subsequence $(x_{\phi(n)},y_{\phi(n)})$ of 
$(x_{n},y_{n})$. More precisely, $\phi(n)$ is a strictly increasing and unbounded sequence of natural numbers. Let us denote by 
 $\lim_{n\rightarrow\infty}(x_{\phi(n)},y_{\phi(n)})$ 
 $=(x_\infty,y_\infty)$ 
 $\in\overline{D_{r}}\setminus D_{\frac{r}{2}}.$ One obtains $f(x_{\phi(n)},y_{\phi(n)})=\frac{1}{\phi(n)}$ and by the continuity of $f$, $f(x_\infty,y_\infty)=\lim_{n\rightarrow\infty}f(x_{\phi(n)},y_{\phi(n)})=\lim_{n\rightarrow\infty}\frac{1}{\phi(n)}=0.$ Therefore, there exists $(x_\infty,y_\infty)\neq (0,0)$, $(x_\infty,y_\infty)\in D_{r}$, such that $f(x_\infty,y_\infty)=0.$ This gives a contradiction with our hypothesis that $(0,0)$ is the isolated local minimum of $f$ in $D_{r}.$ Hence, $\mathcal{C}_\varepsilon$ is bounded. 

In conclusion, we proved that $\mathcal{C}_\varepsilon\cap\partial D_r=\emptyset$. Hence we conclude that the level set $\mathcal{C}_\varepsilon$ could be a disjoint union of connected components which are either diffeomorphic to the segment $[0,1]$, or to the circle $S^1.$ Since we have excluded the closed segments at the first step of our proof, $\mathcal{C}_\varepsilon$ is a disjoint union of connected components which are all diffeomorphic to the circle $S^1.$

\textbf{Fourth and last step: we prove that 
$\mathcal{C}_\varepsilon\stackrel{\text{diffeo}}{\simeq}S^1.$} 

We argue by contradiction. Let us suppose that $\mathcal{C}_\varepsilon$ is a disjoint union of at least two distinct connected components, say $\mathrm{S}_{1}$ and $\mathrm{S}_{2}$, which are both  diffeomorphic to the circle $S^1$. By the Extreme Value Theorem, in the interiors of both these two connected components there is a local extremum of $f$. Since locally the origin is the only critical point in $V$ that $f$ admits by hypothesis, we obtain $(0,0)\in \Int \mathrm{S}_1$ and $(0,0)\in \Int \mathrm{S}_2$. Since $\mathcal{C}_\varepsilon$ is not self-intersecting, the configuration from Figure \ref{fig:unite}(a) below is impossible.

\vspace{-\baselineskip}
\begin{figure}[H]
\centering
\begin{subfigure}[b]{0.45\textwidth}
\includegraphics[scale=0.145,trim={0 0cm 0 6cm}, clip]{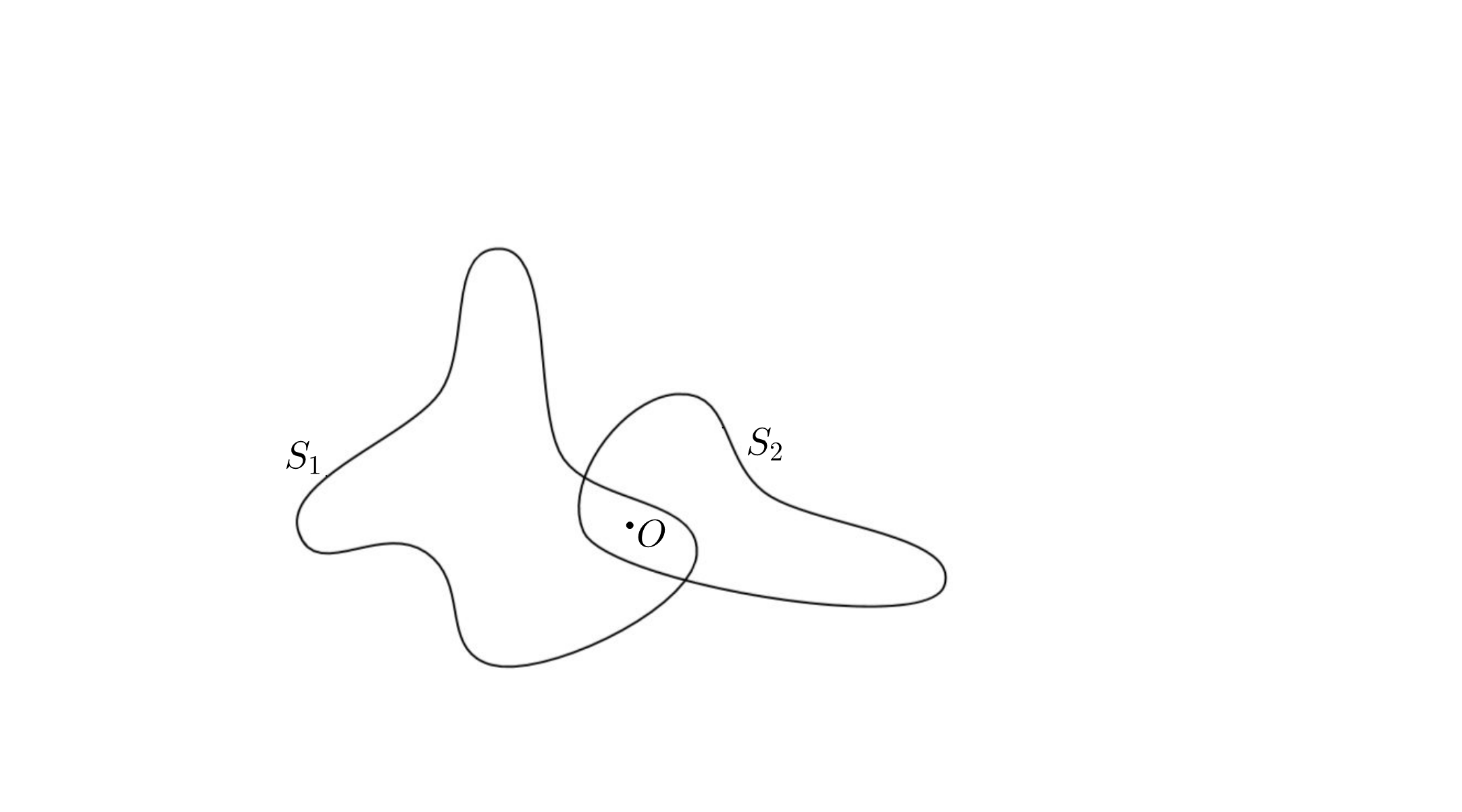} 
\caption{Impossible configuration for $\mathcal{C}_\varepsilon$.\label{fig:config2CercuriImpos}}
\end{subfigure}
\begin{subfigure}[b]{0.45\textwidth}
\includegraphics[scale=0.16,trim={0 0cm 0 6cm}, clip]{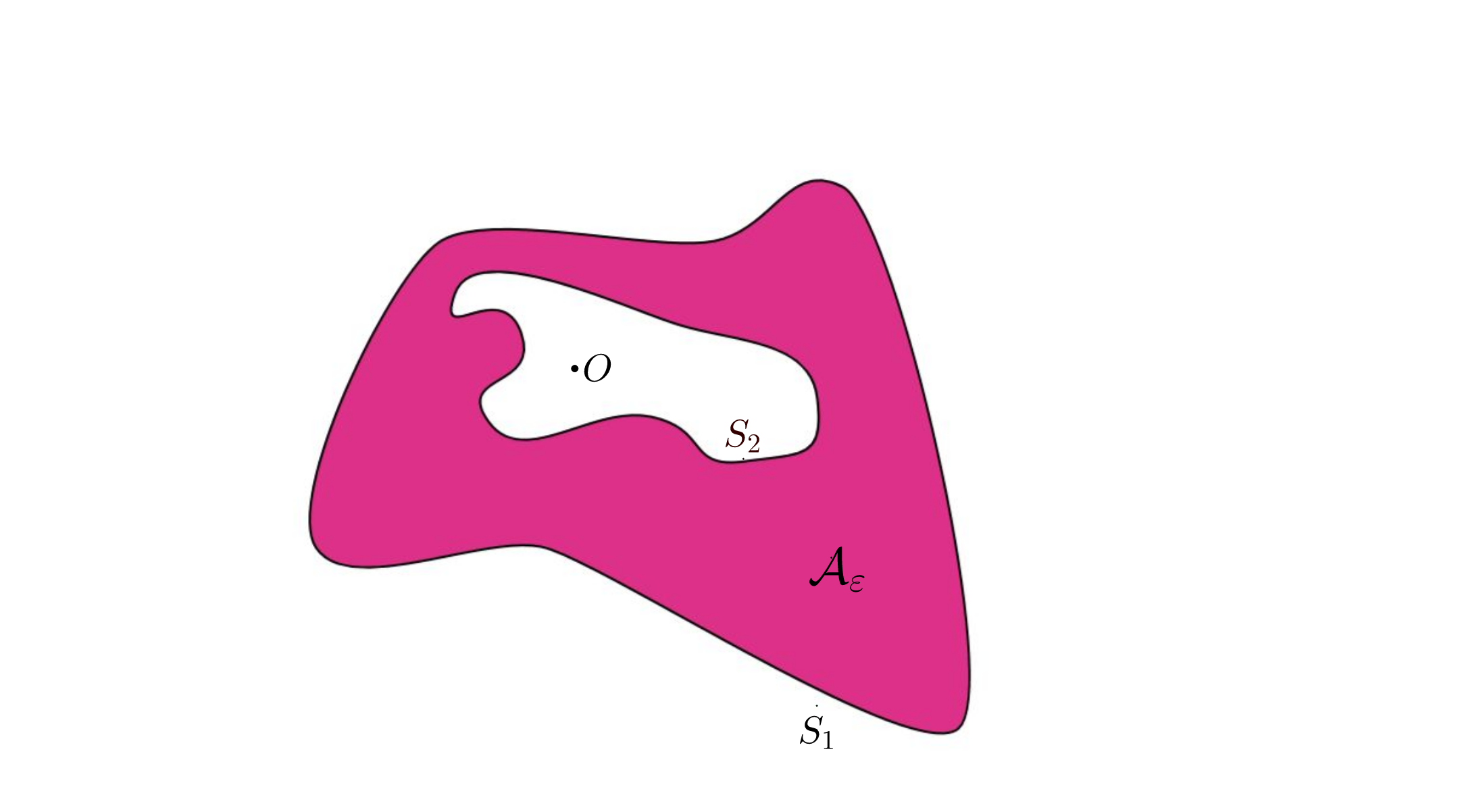} 
\caption{The annulus $\mathcal{A}_\varepsilon$ in pink.\label{fig:2circles}}
\end{subfigure}
\caption{Configurations for $\mathrm{S}_1$ and $\mathrm{S}_2$.\label{fig:unite}}
\end{figure}

 Hence the only two possible situations would be either  $\mathrm{S}_{1}\subset \Int \mathrm{S}_{2}$ or $\mathrm{S}_{2}\subset \Int \mathrm{S}_{1}.$ Without any loss of generality, let us choose $\mathrm{S}_{2}\subset \Int \mathrm{S}_{1}.$ We have $(0,0)\in \Int \mathrm{S}_{2}.$ Let us consider the closed annulus bounded by the two circles, namely $\mathcal{A}_{\varepsilon}:=\mathrm{S}_{1}\cup(\Int \mathrm{S}_{1}\setminus \Int \mathrm{S}_{2}).$ See Figure \ref{fig:unite}(b) above.

Since $\mathcal{A}_{\varepsilon}$ is compact, by the Extreme Value Theorem, the function $f$ must attain a maximum value  and a minimum value on $\mathcal{A}_{\varepsilon}$. We recall that $f$ is not constant, since $(0,0)=f^{-1}(0)$ is a strict local minimum of $f$. If one extremum value is attained by $f$ on $\mathrm{S}_1$ and on $\mathrm{S}_2$ ($f=\varepsilon$ both on $\mathrm{S}_1$ and on $\mathrm{S}_2$), then the other extremum value will necessarily be attained in $\Int \mathcal{A}_\varepsilon$. Hence we obtain a contradiction with the hypothesis of the unique singularity of $f$ in a neighbourhood of the origin. Therefore, we conclude that $\mathcal{C}_\varepsilon$ cannot have more than one connected component diffeomorphic to the circle $S^1$, thus $\mathcal{C}_\varepsilon\stackrel{\text{diffeo}}{\simeq}S^1$ and $(0,0)\in \mathcal{C}_\varepsilon.$

\end{proof}

\subsection{The polar curve}
The notion of a polar curve goes back to the XIXth century, appearing in the work of J.-V. Poncelet (\cite{Pon}) and M. Plücker (\cite{Pl}) in the period of 1813-1830 (see for instance \cite{Te1}, \cite{Te2}, \cite{FT}). Starting with the 1970s, the theory of polar curves\index{polar curve} (see \cite[page 589]{BK}) was renewed by important investigations and
contributions due to L\^ e (\cite{Le1}), Teissier, Merle, Eggers, Delgado, García Barroso (\cite{GB1}, \cite{GB2}, \cite{GB3}), P\l oski, Gwo\' zdziewicz, Casas-Alvero, Michel, Weber, Maugendre (\cite{Mau}), Hefez among others. Considerable research on the polar curves has been carried out in the complex setting. However, less is known about the real polar curves.

\begin{definition}\cite[Section 4]{GB1}\label{DefPolCurve}
Let $f:\bR^2\rightarrow\bR$ be a polynomial function. 
The set 
$$\Gamma(f,x)=\left \{(x,y)\in\bR^2 \mid\frac{\partial f}{\partial y}(x,y)=0\right \}$$
is called \defi{the polar curve of} $f$ \defi{with respect to}\index{polar curve} $x.$
\end{definition}

\begin{remark}

 The polar curve $\Gamma(f,x)$ consists of the points $(x,y)\in\bR^2$ where the level curves of $f$ have a vertical tangent. Note that the vertices of the Poincaré-Reeb tree roughly speaking correspond to points on the level curve with vertical tangent lines. When we study the family of level curves, the points of vertical tangency move along the branches of the polar curve $\Gamma(f,x)$. 

\end{remark}

However, if $f$ has a strict local minimum, then the polar curve does not contain the projection direction (see Proposition \ref{prop:projDirNot}).
\begin{proposition}\label{prop:projDirNot}
If $f:\mathbb{R}^2\rightarrow\mathbb{R}$ has a strict local minimum at the origin and $f(0,0)=0$, then the polar curve $\Gamma(f,x)$ does not contain the line $x=0.$
\end{proposition}
\begin{proof}
We argue by contradiction. If the line $x=0$ is contained in 
      $\Gamma(f,x)$, then we may write $\frac{\partial f}{ \partial  y}(x,y)=xg(x,y)$, where $g\in\mathbb{R}[x,y].$ Then $f(x,y)=f(x,0)+\int_{0}^y xg(x,t)\mathrm{d}t=x\varphi(x,y),$ since $f(0,0)=0$. Hence $f$ does not have a strict local minimum at the origin, because it vanishes on $x=0$. Contradiction. 
\end{proof}

\begin{example}
In Figure 2.22 is represented the polar curve $\Gamma(f,x)$ of Coste’s example (see Example \ref{ex:Coste}: $f(x,y):=x^2+(y^2-x)^2$). It has two irreducible components, with equations $(y=0)\cup (y^2-x=0)$.
 
\begin{figure}[H]
\centering
\includegraphics[scale=0.3]{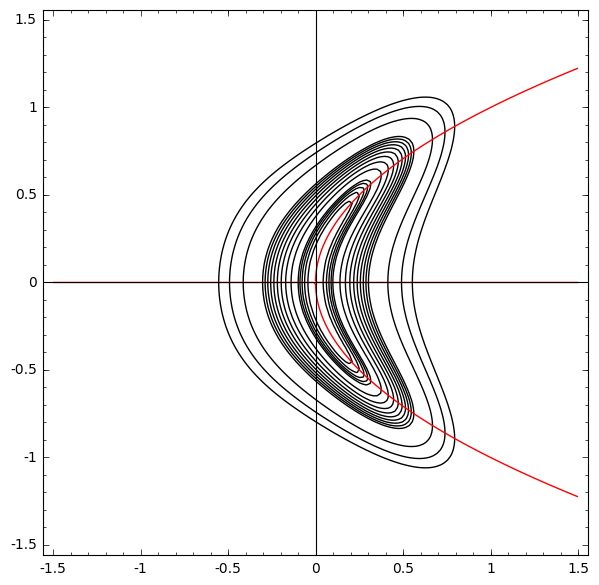} 
\caption{The polar curve $(y=0)\cup (y^2-x=0)$ for Example \ref{ex:Coste}, namely Coste's example\label{fig:polarBanana}}
\end{figure}
\end{example}

\begin{example}
Recall our Example \ref{ex:unu}: $f(x,y):=x^{16}+(y^2+x)^2 (y^2-x)^2.$ In Figure \ref{fig:polarBone}, the polar curve $\Gamma(f,x)$ has three components $(y=0)\cup (y^2+x=0) \cup (y^2-x=0)$.
 
\begin{figure}[H]
\centering
\includegraphics[scale=0.3]{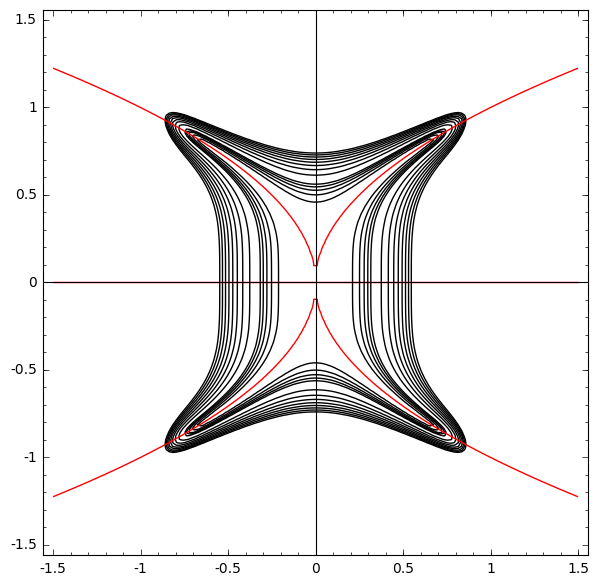} 
\caption{The polar curve $(y=0)\cup (y^2+x=0) \cup (y^2-x=0)$ for Example \ref{ex:unu}.\label{fig:polarBone}}
\end{figure}
\end{example}

\begin{example}\label{zero}
We can create new examples starting from Coste's example by taking the parametrisation of a branch $\gamma$ of the polar curve $\Gamma(f,x)$, namely  $t\mapsto(t^2,t)$, and modify it, as follows. 

Consider the new parametrisation $(t\mapsto p(t),q(t))$, where $p(t)=t^2+t^3$ and $q(t):=t$. Let us compute the resultant of the polynomials $g(t):=x-p(t)$ and $h(t):=y-q(t)$, in order to find the implicit polynomial whose zero locus is $\gamma.$ We have the resultant (see for instance \cite[page 178]{BK}) $$\mathrm{Res}_t(x,y)=\det \begin{bmatrix}
    -1 & -1 & 0 & x \\
    -1 & y & 0 & 0 \\
    0 & -1 & y & 0 \\
    0 & 0 & -1 & y 
\end{bmatrix}=x-y^3-y^2.$$

Thus, (see \cite[page 181]{BK}) we have a new algebraic branch $\tilde{\gamma}:=\{(x,y)\in\mathbb{R}^2\mid x-y^3-y^2=0\}.$

Take now $\tilde{f}(x,y):=x^2+(x-y^3-y^2)^2(x-y^2)^2$ and obtain the local situation of the polar curve in a small enough neighbourhood of the origin as presented in Figure \ref{fig:deform} below. 

\begin{figure}[H]
\centering
\includegraphics[scale=0.07]{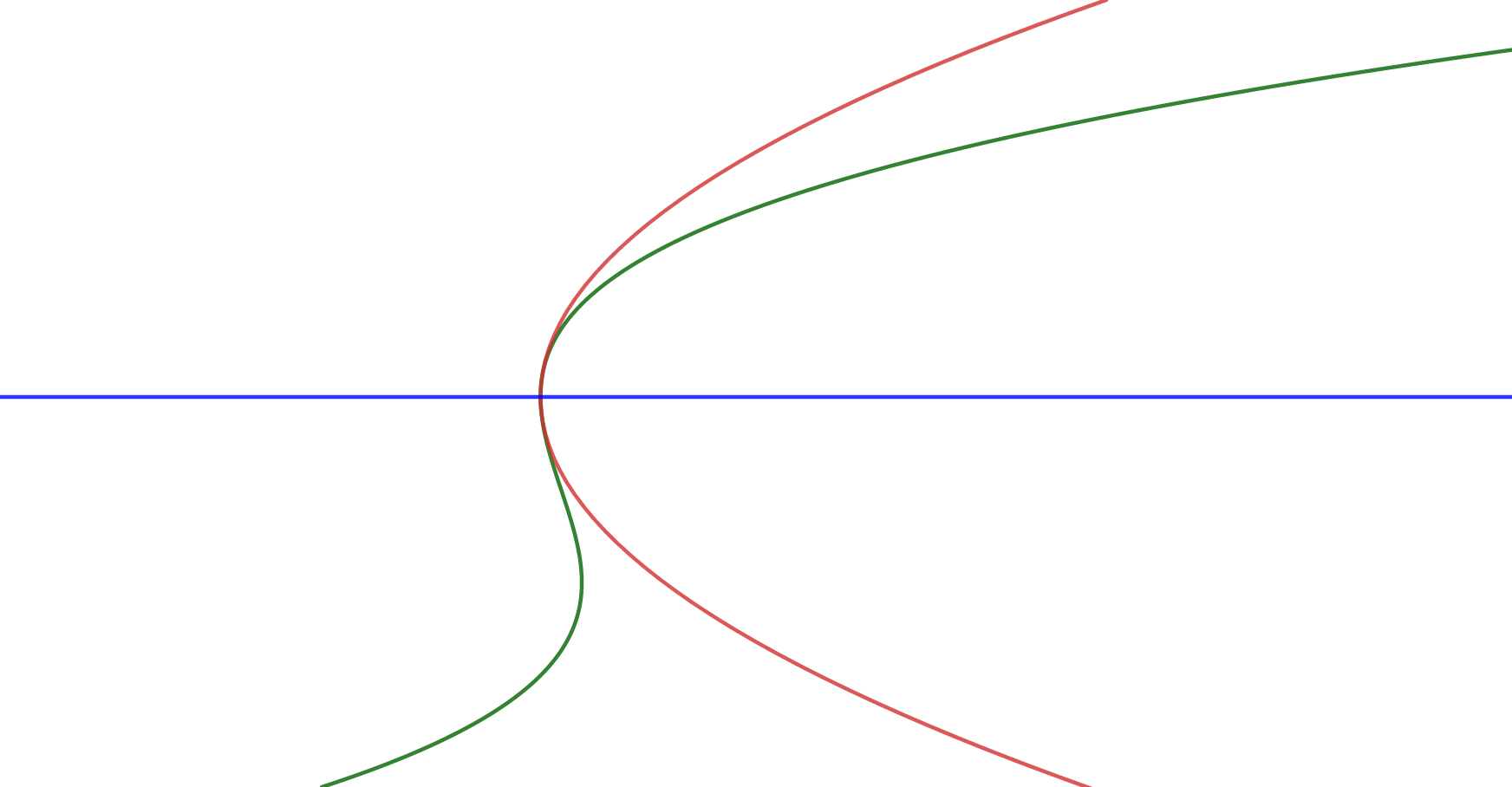} 
\caption{We modified the initial parabola $\gamma$ (in red) into $\tilde{\gamma}$ (in green), both branches of the polar curve $\Gamma(\tilde{f},x)$, to obtain a different local shape of $\mathcal{C}_\varepsilon$ when compared to the initial one in Coste's example.\label{fig:deform}}
\end{figure}
\end{example}

\begin{example}
Recall our Example \ref{ex:doi}: $f(x,y):=x^{6}+(y^4+y^2-x)^2 (y^2-x)^2$. In Figure \ref{fig:polarDoubleBanana}, the polar curve $\Gamma(f,x)$ has four components.
 
\begin{figure}[H]
\centering
\includegraphics[scale=0.3]{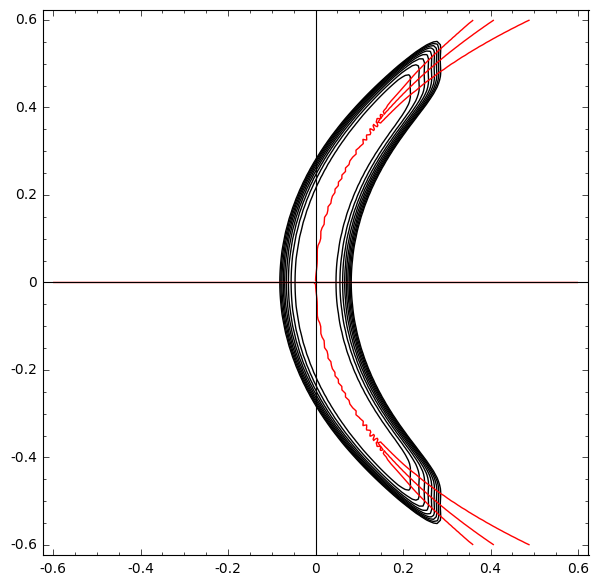} 
\caption{The polar curve for our Example \ref{ex:doi}.\label{fig:polarDoubleBanana}}
\end{figure}
\end{example}

For more details about Lemma \ref{lemma:halfBr} below, we refer to \cite[page 105]{Gh1}.

\begin{lemma}\cite{Mi2}\label{lemma:halfBr}
Let $\Gamma$ be a real algebraic curve. Firstly, there exists a Milnor disk\index{Milnor disk} for $ \Gamma $ at the origin: a Euclidean disk
       such that the boundaries of all the concentric disks inside the Milnor disk are transverse to the curve. Secondly, if the curve is analytically irreducible at the origin, then in such a disk the curve is homeomorphic to a segment. 
\end{lemma}

A proof can be found in \cite[Lemma 3.3, page 28]{Mi2} and \cite[page 105]{Gh1}. For recent progress in the complex setting on the intersection between plane curves and Milnor disks, see the very recent work of Bodin and Borodzik (\cite{Bod1}).

\begin{definition}\cite[page 2]{Gh1}
A \defi{real branch}\index{real branch} is the germ at the origin of an analytically irreducible curve at this point.
\end{definition}

A definition similar to Definition \ref{def:halfBrGeneral} can be found, for instance, in \cite[page 97]{Cas}.

\begin{definition}\label{def:halfBrGeneral}
A \defi{half-branch}\index{real branch!half-branch} (see Figure \ref{fig:halfBr} below) is the germ at the origin of the closure of one of the connected components of the complementary of the origin in a Milnor representative of a branch.

\end{definition}

\begin{remark}
We can choose to denote one half-branch by $\gamma_i^+$ and the other one by $\gamma_i^-$.
\end{remark}

Definition \ref{def:halfBr} is inspired by \cite[Section 2]{CP}.

\begin{definition}\label{def:halfBr}
We call \defi{polar half-branches}\index{polar half-branch} the half-branches of the polar curve $\Gamma:=\Gamma(f,x)$, in the sense of Definition \ref{def:halfBrGeneral}. Let $\gamma^*\subset\Gamma$ be a polar half-branch and let $V$ be a good neighbourhood of the origin, as in Definition \ref{DefGoodNhb}. We say that $\gamma^*$ in $V$ is a \defi{right polar half-branch}\index{real branch!half-branch!right} if $x_{|\gamma^*}\geq 0.$ If $x_{|\gamma^*}\leq 0,$ we say that $\gamma^*$ in $V$ is a \defi{left polar half-branch}.\index{real branch!half-branch!left}
\end{definition}

\begin{remark}
By Corollary \ref{cor:xStrictIncreasing}, in a sufficiently small $V$ we always have that $x$ is strictly increasing when restricted to the right polar half-branches and strictly decreasing when restricted to the left polar half-branches.
\end{remark}

\begin{remark}
We sometimes call the union of the right polar half-branches the \textbf{positive} polar curve.
\end{remark}

\begin{figure}[H]
\centering
\includegraphics[scale=0.11]{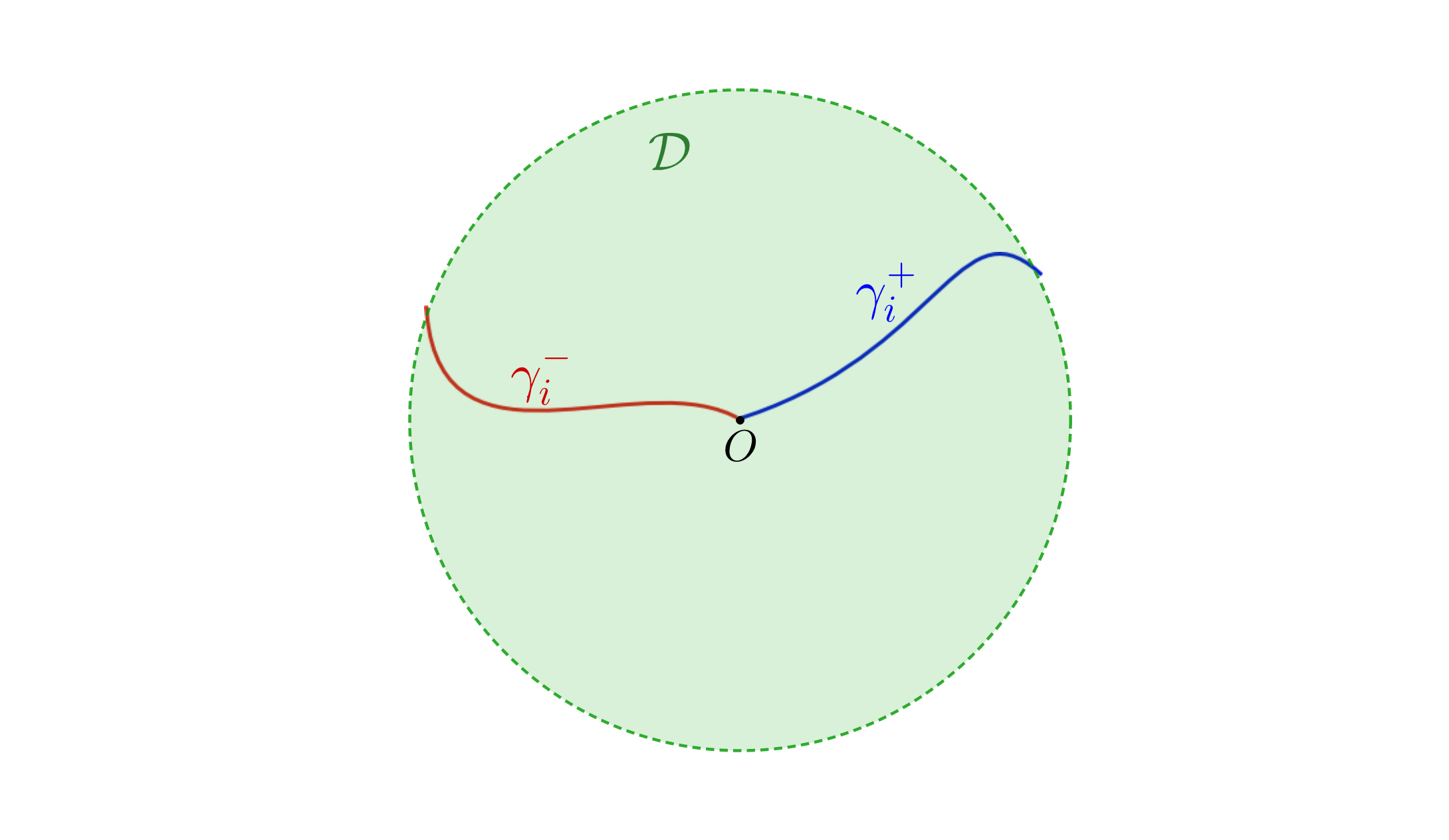} 
\caption{A branch $\gamma_i$ of $\Gamma$ has exactly two half-branches: $\gamma_i^{-}$ and $\gamma_i^{+}$.\label{fig:halfBr}}
\end{figure}

A result similar to the following lemma can be found in \cite[Lemma 1.1]{FP}.
\begin{lemma}\label{LemmaStrictIncreasing}
Let $g:\bR^2\rightarrow \bR$ be a polynomial function with a strict local minimum at the origin $O=(0,0)$ and let $\Gamma$ be an algebraic curve passing through the origin. In a small enough neighbourhood $V$ of the origin, for any  half-branch $\gamma^*\subset\Gamma$, the restriction $g_{|\gamma^*}$ is strictly monotone. In particular, in the sufficiently small neighbourhood $V$, for a small enough $0<\varepsilon\ll 1,$ the intersection $(g=\varepsilon)\cap \gamma^*$ consists of exactly one point, for each half-branch $\gamma^*$ (see Figure \ref{fig:1Pt}). Here the symbol $*\in\{+,-\}$.
\end{lemma}

\begin{figure}[H]
\centering
\includegraphics[scale=0.1,trim={0 6cm 0 4cm}, clip]{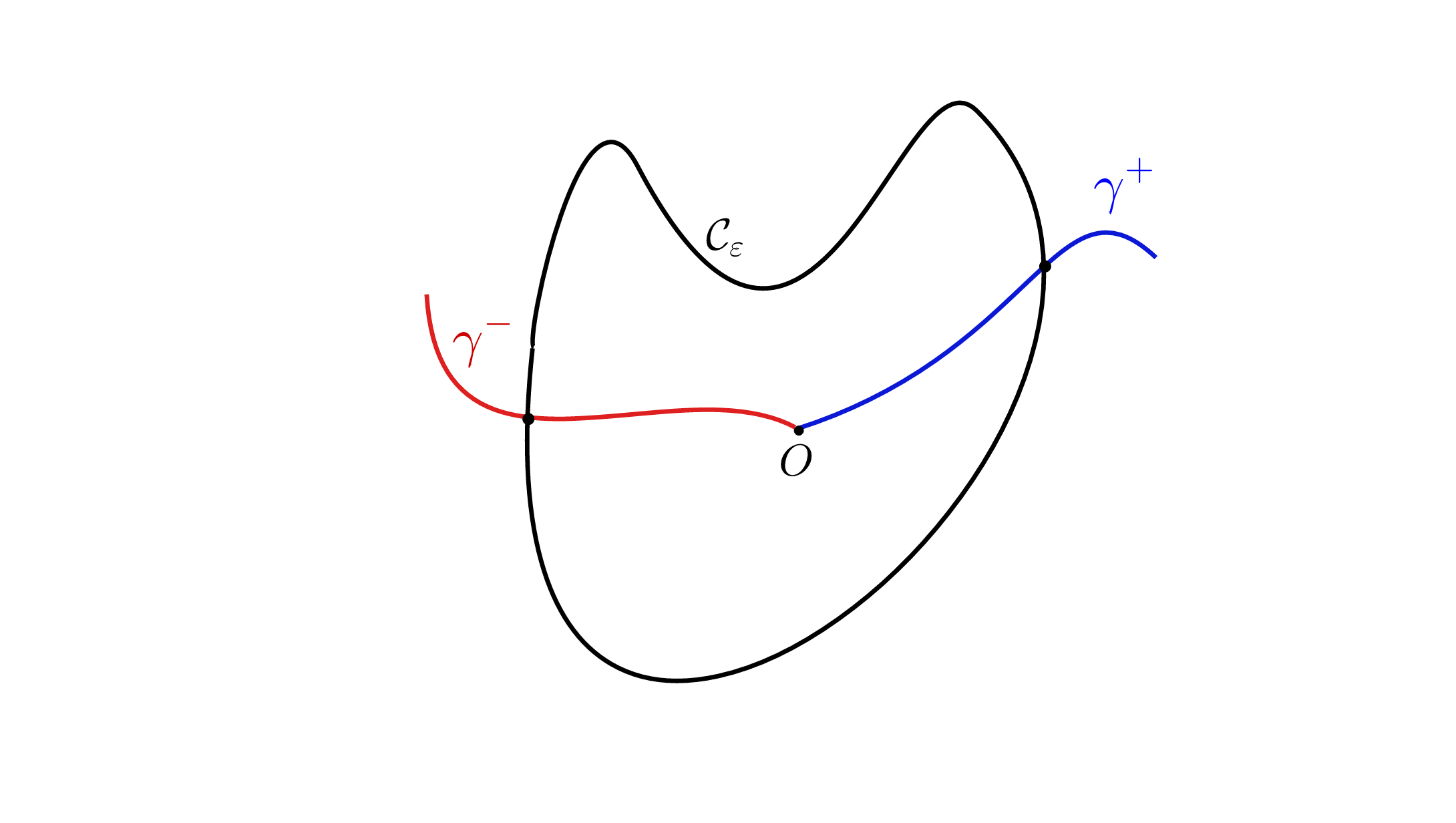} 
\caption{For a small enough neighbourhood $V$ and for a sufficiently small $0<\varepsilon\ll 1$, the intersection $\mathcal{C}_\varepsilon\cap \gamma^-$ (respectively $\mathcal{C}_\varepsilon\cap \gamma^+$) consists of exactly one point.\label{fig:1Pt}}
\end{figure}

\begin{proof}
Let us choose, without loss of generality, to study the behaviour of $g$ on $\gamma^+.$ Suppose that $g_{|\gamma^+}$ is not constant. Since $\Gamma$ is an algebraic curve, it has only a finite number of critical points. Therefore one can assume that in a small enough neighbourhood $V$ of the origin $O=(0,0)$, the half-branch $\gamma^+$ has no critical points except maybe the origin $O=(0,0)$. By \cite[Lemma 2.7]{Mi2}, applied to the polynomial function $g$ and to the one-dimensional manifold $\gamma^+$, which is a subset of the algebraic curve $\Gamma$, in the sufficiently small neighbourhood of $O=(0,0)$, the set of the local extrema of 
$g_{|\gamma^+}$ is an algebraic set of dimension zero. Thus, by \cite[Corollary 9.1, page 227]{Ei}, the set of local extrema is finite. Now let us choose $V$ sufficiently small such that 
in $V$ there are no local extrema of the function $g_{|\gamma^+\setminus \{(0,0)\}}.$ Thus, in  $V$, $g_{|\gamma^+}$ is  strictly monotone.
\end{proof}

\begin{corollary}\label{RemarkStrictlyIncr}
Given a polynomial function $f$ with a strict local minimum at the origin and its polar curve $\Gamma(f,x)$, if $\gamma^+$ is a right polar half-branch, then $f_{\gamma^+}$ is always strictly increasing in $V$, when going further from the origin.\index{polar curve}
\end{corollary}

\begin{proof}
Apply Lemma \ref{LemmaStrictIncreasing} for the function $f$ and the polar curve $\Gamma(f,x),$ taking into account that $(0,0)$ is a strict local minimum.
\end{proof}

\begin{example}
Let us present in Figure \ref{fig:fPasCroiss} a configuration which is \textbf{impossible} asymptotically (i.e. for $0<\varepsilon\ll 1$ sufficiently small) in $V$: even if $x_{|\gamma^+}$ is strictly increasing, we have $f_{|\gamma^+}$ not strictly increasing. 

\vspace{-\baselineskip}
\begin{figure}[H]
\centering
\includegraphics[scale=0.12,trim={0 4cm 0 6cm}, clip]{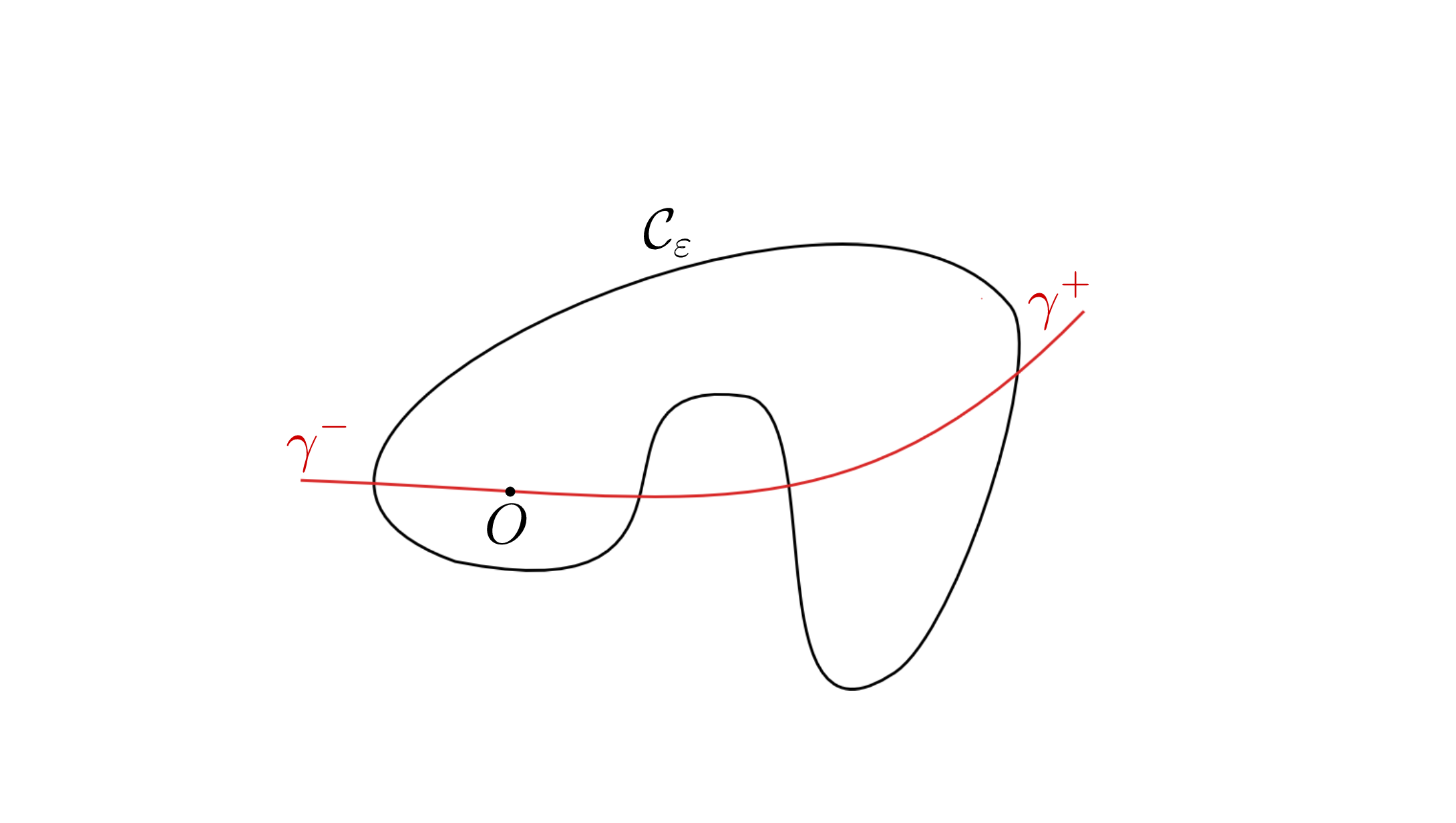} 
\caption{Locally impossible configuration in $V$: $f_{|\gamma^+}$ not strictly increasing.\label{fig:fPasCroiss}}
\end{figure}
\end{example}

\begin{corollary}
In particular, the shape of $\mathcal{C}_\varepsilon$ in a sufficiently small enough $V$ cannot go backwards or be spiral shaped.
\end{corollary}

\begin{corollary}\label{Cor:NoSpiral}
We can choose a sufficiently small neighbourhood $V$ such that the intersection $\gamma^*\cap \mathcal{C}_\varepsilon$ consists of exactly one point, where $\mathcal{C}_\varepsilon:=(f=\varepsilon)\cap V$.
\end{corollary}

\begin{proof}
By Corollary \ref{RemarkStrictlyIncr}, for any polar half-branch $\gamma^*$ one has $f_{\gamma^*}$ is locally strictly increasing, thus for sufficiently small $\varepsilon,$ the intersection $(f=\varepsilon)\cap\gamma^*$ consists of at most one point. 
\end{proof}

A consequence of Lemma \ref{LemmaStrictIncreasing}, applied to the polar curve $\Gamma(f,x)$ of $f$ and for the function $g:\mathbb{R}^2\rightarrow\mathbb{R}$, $g(x,y):=x$ is the following result.

\begin{corollary}\label{cor:xStrictIncreasing}
Let $f$ be a polynomial function with a strict local minimum at the origin and let $\Gamma(f,x):=\left \{(x,y)\in\mathbb{R}^2\;\middle|\; \frac{\partial f}{\partial y}(x,y)=0\right \}$ be its polar curve with respect to $x$. In a small enough neighbourhood $V$ of the origin, for any polar half-branch $\gamma^*\subset\Gamma$, the restriction $x_{|\gamma^*}$ is strictly monotone. Here the symbol $*\in\{+,-\}$. 
\end{corollary}

\begin{proof}
Since $x$ is not constant on $\Gamma(f,x)$ by Proposition \ref{prop:projDirNot}, we can apply Lemma \ref{LemmaStrictIncreasing}.
\end{proof}

\begin{example}
Let us present in Figure \ref{fig:xPasCroiss} a locally \textbf{impossible} configuration in $V$ of $\mathcal{C}_\varepsilon$ with a polar half-branch $\gamma^+$: even if $f_{|\gamma^+}$ is strictly increasing, we have $x_{|\gamma^+}$ not strictly increasing. 

\begin{figure}[H]
\centering
\includegraphics[scale=0.12,trim={0 6cm 0 4cm}, clip]{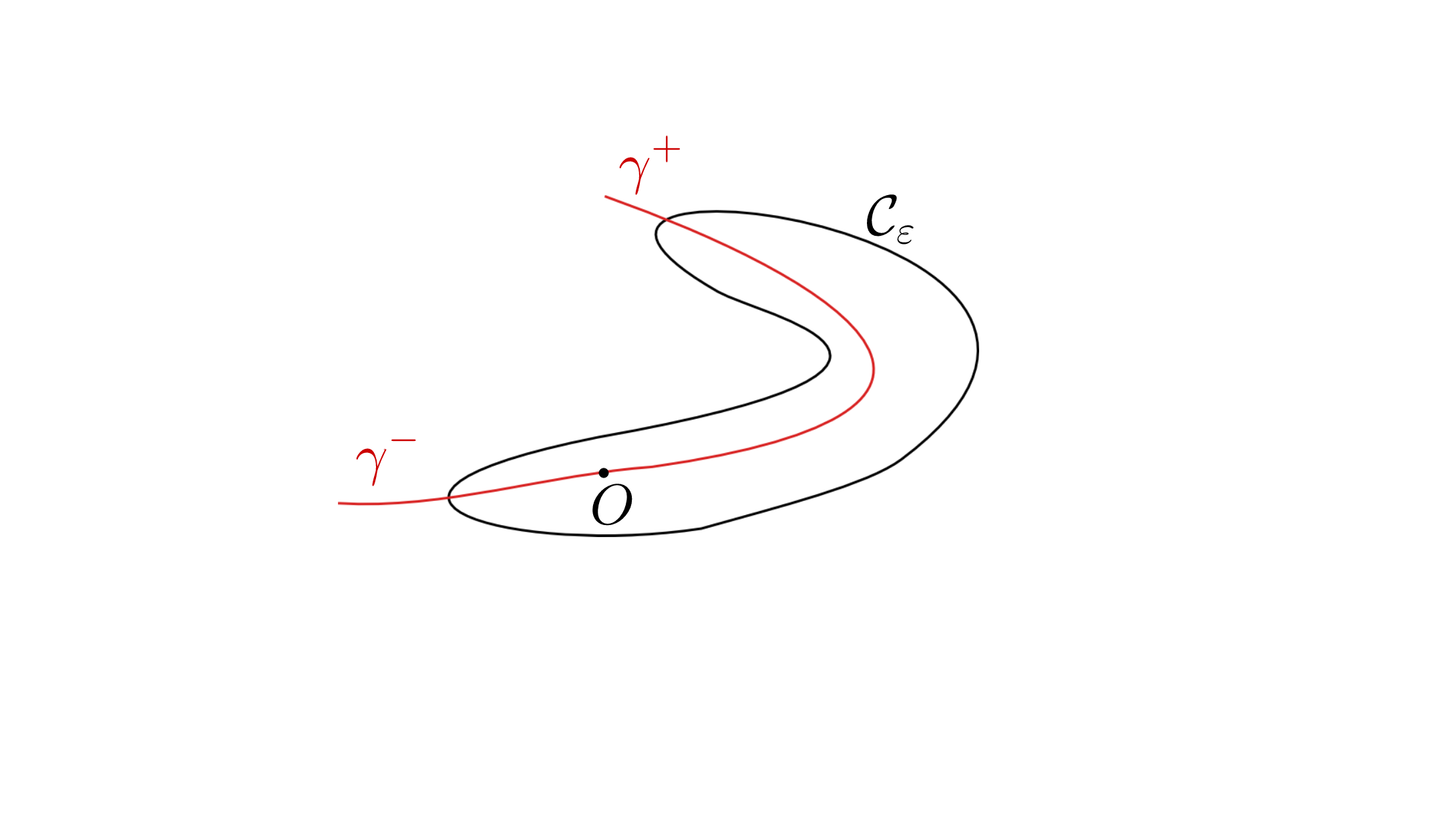} 
\caption{Impossible asymptotic configuration in $V$, since $x_{|\gamma^+}$ is not strictly increasing.\label{fig:xPasCroiss}}
\end{figure}
\end{example}

\begin{remark}\label{remark:noVertTg}
In particular, the shape of $\mathcal{C}_\varepsilon$ in a sufficiently small $V$ cannot \enquote{go backwards}. More rigorously, if we choose $V$ small enough, then the polar curve $\Gamma(f,x)$ has no vertical tangents in $V$. A vertical tangent of a polar half-branch $\gamma^*$ would correspond to a local extremum of the function $x_{|\gamma^*}$, which is impossible.
\end{remark}

A consequence of Lemma \ref{LemmaStrictIncreasing}, applied to the polar curve $\Gamma(f,x)$ of a polynomial function $f$ with a strict local minimum at the origin and to the function square of the distance to the origin, namely $x^2+y^2$ is the following result.

\begin{corollary}\label{cor:distantaMonotone}
Let $f$ be a polynomial function $f$ with a strict local minimum at the origin and let $\Gamma(f,x)$ be its polar curve with respect to $x$. In a small enough neighbourhood $V$ of the origin, for any polar half-branch $\gamma^*\subset\Gamma$, the restriction ${x^2+y^2}_{|\gamma^*}$ is strictly monotone. Here the symbol $*\in\{+,-\}$.
\end{corollary}

\begin{corollary}\label{cor:onlyOnePoint}
We can choose a sufficiently small neighbourhood $V$ such that for any polar half-branch $\gamma^*$, the intersection $\gamma^*\cap \partial V$ consists of exactly one point.
\end{corollary}

\subsection{Stabilisation. No spiralling}\label{sect:asymptReeb}

\begin{definition}\label{DefGoodNhb}
Let us consider a polynomial function $f:\mathbb{R}^2\rightarrow\mathbb{R}$ with a strict local minimum at the origin, such that $f(0,0)=0$. Denote by $\Gamma(f,x)$ the polar curve of $f$ with respect to the direction $x$. Let $V$ be a neighbourhood of the origin $O(0,0)$. We say that the neighbourhood $V$ is a \defi{good neighbourhood of the origin for the couple $(f,x)$}\index{good neighbourhood} if $V$ satisfies:

(i) the point $(0,0)$ is \textbf{the only strict local minimum} of $f$ in $V$;

(ii) we have $V\cap (f=0)=\{(0,0)\}.$

(iii) $\Sing (\Gamma(f,x))\cap V \subseteq \{(0,0)\};$

(iv) $V$ is small enough such that on each polar half-branch $\gamma^*$ of $\Gamma(f,x)$ we have $f_{|\gamma^*}$ is strictly monotone;

(v) $V$ is small enough such that on each polar half-branch $\gamma^*$ of $\Gamma(f,x)$ we have $x_{|\gamma^*}$ is strictly monotone.

(vi) $V$ is small enough such that for any polar half-branch $\gamma^*$, the intersection $\gamma^*\cap \partial V$ consists of exactly one point

\end{definition}

\begin{remark}\label{rem:nuSeIntersect}
Condition (i) implies that for any two distinct polar half-branches $\gamma_i^*$ and $\gamma_j^*$, we have $$(\gamma_i^*\cap\gamma_j^*)\cap V=\{(0,0)\}.$$
\end{remark}

\begin{definition}\label{def:asymptPReeb}
Consider a polynomial function $f$ with a strict local minimum at the origin, such that $f(0,0)=0$. The \defi{asymptotic Poincaré-Reeb tree}\index{Poincaré-Reeb tree!asymptotic} of $f$ relative to $x$ is the \enquote{limit} Poincaré-Reeb tree of its sufficiently small levels near the origin with respect to $x$. Here by \enquote{limit} we mean: for a sufficiently small $\varepsilon >0.$
 
\end{definition}

\begin{theorem}\label{th:asymptoticPReeb}
Besides the properties of the Poincaré-Reeb tree (see Corollary \ref{cor:PReebInGeneral}), \textbf{the asymptotic Poincaré-Reeb tree} stabilises up to equivalence (see Definition \ref{def:transversalTree}) when the level gets small enough. It is a \textbf{rooted} plane tree, the root being the image of the origin. The total preorder relation on its vertices, induced by the function $x$, is \textbf{strictly monotone} on each geodesic starting from the root. In addition, the asymptotic Poincaré-Reeb tree is a union of \textbf{a positive tree} (at the right of the origin) and \textbf{a negative tree} (at the left of the origin), the latter two having a common root, that is the image of the origin.
\end{theorem}

\vspace{-\baselineskip}
\begin{proof}

Let us first prove the stabilisation. Take two right polar half-branches $\gamma_i$ and $\gamma_j$. Denote by $(x_i(t),y_i(t))$, respectively $(x_j(t),y_j(t))$ their corresponding Newton-Puiseux parametrisations (see \cite[Theorem 2.1.1]{Wa}). Here $x_1(t), y_1(t), x_2(t), y_2(t)\in\mathbb{R}\{t\}$ are convergent analytic parametrisations. Now denote by $g(t):=x_i(t)- x_j(t)$. We have two possibilities: either $g(t)\equiv 0,$ i.e. we have a vertical bitangent, or $g(t)\not\equiv 0.$ In the second case, by taking a sufficiently small $t$, the sign of the analytic function $g$ does not change. Since there are finitely many polar half-branches, we can choose a sufficiently small $t$ such that the total preorder stabilises.

In the asymptotic setting, if we consider a polynomial function $f:\mathbb{R}^2\rightarrow\mathbb{R}$ with a strict local minimum at the origin such that $f(0,0)=0$, then the root of the Poincaré-Reeb tree is the image of the origin by the quotient map (see Section \ref{sect:ConstructionPR}). By Lemma \ref{lemma:origInInterior}, the origin is in the interior of the disk $\mathcal{D}_\varepsilon$. Since the real plane is cooriented, we distinguish without ambiguity between the  positive tree (at the right of the origin) and the negative tree (at the left of the origin), the latter two having the common root, that is the image of the origin.

Moreover, we have the induced application $\tilde{x}:=x$ of the projection function $x:\mathbb{R}^2\rightarrow\mathbb{R}$ to the embedded Poincaré-Reeb tree. By Corollary \ref{cor:xStrictIncreasing}: for any polar half-branch $\gamma^*\subset\Gamma(f,x)$, the restriction $x_{|\gamma^*}$ is strictly monotone. Thus, $\tilde{x}$ has no internal local extremum on an edge of $\mathcal{R}(f,x).$ To the right, $\tilde{x}$ only attains its local maxima on the leaves of the Poincaré-Reeb tree. Similarly, to the left, $\tilde{x}$ only attains its local minima on the leaves. In other words, there are no edges which \enquote{go backwards} towards the root. 

\end{proof}

\begin{remark}

 The strict monotonicity on the geodesics starting from the root implies that the small enough level curves $\mathcal{C}_\varepsilon$ have no turning back or spiralling phenomena. In particular, for $0<\varepsilon\ll 1$ sufficiently small, there are no shapes $\mathcal{C}_\varepsilon$ like the one in Figure \ref{fig:spiraling}.

\begin{figure}[H]
\centering
\includegraphics[scale=0.125]{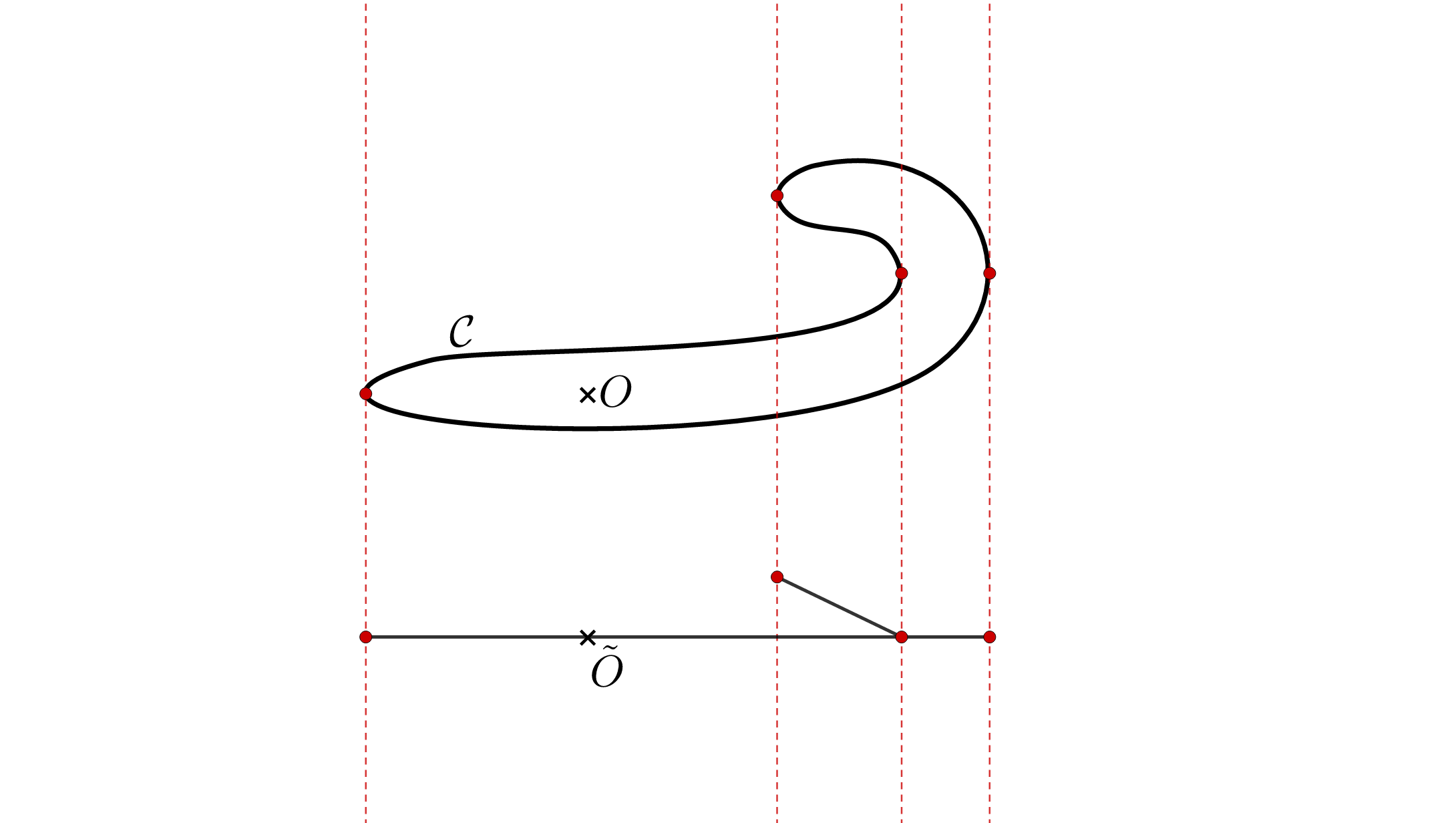} 
\caption{Impossible shape for small enough $0<\varepsilon\ll 1$.\label{fig:spiraling}}
\end{figure}
\end{remark}

\newpage
\printbibliography

\end{document}